\documentclass[a4paper, 10pt, notitlepage]{article}

\usepackage{amsthm, amsmath, amssymb, latexsym}
\usepackage{mathrsfs,cite}
\usepackage[multiple]{footmisc}
\usepackage{graphicx}

\theoremstyle{plain}
\newtheorem{theorem}{Theorem}

\newtheorem{proposition}{Proposition}

\theoremstyle{definition}
\newtheorem{definition}{Definition}
\newtheorem{example}{Example}

\theoremstyle{remark}
\newtheorem{remark}{Remark}

\DeclareMathOperator{\Gr}{Graph}
\DeclareMathOperator{\co}{co}
\DeclareMathOperator{\sign}{sign}
\DeclareMathOperator{\svsign}{Sign}
\DeclareMathOperator{\linhull}{span}
\DeclareMathOperator{\cl}{cl}
\DeclareMathOperator{\determ}{det}
\DeclareMathOperator{\dimension}{dim}

\author{Dolgopolik M.V.\footnote{Institute for Problems in Mechanical Engineering, Russian Academy of Sciences, Saint
Petersburg, Russia. e-mail: maxim.dolgopolik@gmail.com}}
\title{Metric Regularity of Quasidifferentiable Mappings and Optimality Conditions for Nonsmooth Mathematical
Programming Problems}

\begin{document}

\maketitle

\begin{abstract}
This article is devoted to the analysis of necessary and/or sufficient conditions for metric regularity in terms of 
Demyanov-Rubinov-Polyakova quasidifferentials. We obtain new necessary and sufficient conditions for the local metric
regularity of a multifunction in terms of quasidifferentials of the distance function to this multifunction. We also
propose a new MFCQ-type constraint qualification for a parametric system of quasidifferentiable equality and inequality
constraints and prove that it ensures the metric regularity of a multifunction associated with this system. As an
application, we utilize our constraint qualification to strengthen existing optimality conditions for 
quasidifferentiable programming problems with equality and inequality constraints. We also prove the independence of the
optimality conditions of the choice of quasidifferentials and present a simple example in which the optimality
conditions in terms of quasidifferentials detect the non-optimality of a given point, while optimality conditions in
terms of various subdifferentials fail to disqualify this point as non-optimal.
\end{abstract}

\section{Introduction}

Metric regularity plays a very important role in various parts of optimization theory and numerical analysis,
including stability analysis of perturbed optimization problems, subdifferential calculus, analysis of optimality
conditions etc. \cite{BonnansShapiro,Ioffe,Mordukhovich_I,Mordukhovich_II,Auslender,Borwein,Ioffe_book} Necessary
and/or sufficient conditions for metric regularity are usually expressed in terms of various slopes,
subdifferentials and coderivatives \cite{Cominetti,Aze,Ioffe,Mordukhovich_I,Ioffe_book}. However, if one studies
nonsmooth problems with \textit{quasidifferentiable} data and wants to utilize quasidifferential calculus
\cite{DemRub_book,DemRub_collection}, these conditions for metric regularity become very inconvenient, since one has to
compute and use subdifferentials/coderivatives and quasidifferentials simultaneously. In this case it seems more
reasonable to apply necessary and/or sufficient conditions for metric regularity in terms of quasidifferentials. Such
conditions were studied by Uderzo in \cite{Uderzo,Uderzo2,Uderzo2007}.

One of the main goals of this paper is to improve the main results of \cite{Uderzo,Uderzo2,Uderzo2007} and obtain simple
conditions for metric regularity in terms of quasidifferentials. With the use of general results on metric regularity
\cite{Ioffe} we obtain new necessary and sufficient conditions for the metric regularity of multifunctions in terms of
quasidifferentials of the distance function to this multifunction (see~\cite{Dudov} for some results on the
quasidifferentiability of this function). These conditions significantly generalize and improve some results from
\cite{Uderzo}. For example, our conditions, unlike the ones in \cite{Uderzo}, are invariant under the choice of
quasidifferentials. However, both our conditions and the ones in \cite{Uderzo,Uderzo2,Uderzo2007} have a significant
drawback. Namely, one must verify the validity of certain inequalities in a neighbourhood of a given point to apply
these conditions. To overcome this issue, we introduce a new MFCQ-type constraint qualification for a parametric system
of quasidifferentiable equality and inequality constraint and demonstrate that this constraint qualification guarantees
the local metric regularity of a multifunction associated with this system (see~\cite{KuntzScholtes} for a discussion of
constraint qualifications for quasidifferentiable optimization problems with inequality constraints). 

As an application, we utilize our constraint qualification to obtain new necessary optimality conditions for
quasidifferentiable programming problems with equality and inequality constraints that strengthen existing optimality
conditions for these problems in terms of quasidifferentials \cite{Shapiro84,Shapiro86,Polyakova} (optimality conditions
for such problems involving, e.g. the Demyanov difference of quasidifferentials, can be found in \cite{Gao2000}). We
prove the independence of our optimality conditions of the choice of quasidifferentials
(cf.~\cite{Luderer,LudererRosigerWurker}) and present a simple example in which our optimality conditions detect the
non-optimality of a given point, while optimality conditions in terms of Clarke, Michel-Penot, Jeyakumar-Luc, Ioffe and
Mordukhovich subdifferentials fail to disqualify this point as non-optimal.

The paper is organized as follows. Necessary and sufficient conditions for metric regularity of multifunctions
in terms of quasidifferentials are obtained in Section~\ref{Section_MetricRegularity}. In this section, we also
introduce a new MFCQ-type constraint qualification for parametric systems of quasidifferentiable equalities and
inequalities and study its connection with metric regularity. This constraint qualification is applied to the 
derivation of new optimality conditions for quasidifferentiable programming problems in 
Section~\ref{Section_OptimalityConditions}. Finally, some basic definitions and facts from quasidifferential calculus
are collected in Section~\ref{Section_Preliminaries}.

\section{Quasidifferentiable Functions}
\label{Section_Preliminaries}

From this point onwards let $X$ be a real Banach space. Its topological dual space is denoted by $X^*$, whereas the
canonical duality pairing between $X$ and $X^*$ is denoted by $\langle \cdot, \cdot \rangle$. The zero vector of a
vector space $Y$ is denoted by $\mathbb{O}_Y$ or simply by $\mathbb{O}$ when the underline space is clear from the
context.

Let $U \subset X$ be an open set. Recall that a function $f \colon U \to \mathbb{R}$ is called Dini (Hadamard)
directionally differentiable at a point $x \in U$, if for any $h \in X$ there exists the finite limit
\begin{align*}
  f'_D(x, h) &= \lim_{\alpha \to +0} \frac{f(x + \alpha h) - f(x)}{\alpha} \\
  \bigg( f'_H(x, h) &= \lim_{[\alpha, h'] \to [+0, h]} \frac{f(x + \alpha h') - f(x)}{\alpha} \bigg)
\end{align*}
(see~\cite{Giannessi} for a discussion about the limit in the definition of Hadamard directional derivative). Clearly,
if $f$ is Hadamard directionally differentiable at $x$, then $f'_H(x, \cdot) = f'_D(x, \cdot)$. Therefore, it is
natural to refer simply to the directional derivative of $f$ at $x$ and denote it by $f'(x, \cdot)$.

\begin{definition}
A function $f \colon U \to \mathbb{R}$ is called Dini (Hadamard) \textit{quasidifferentiable} at a point $x \in U$ if
$f$ is Dini (Hadamard) directionally differentiable at $x$, and its directional derivative can be represented as the
difference of two continuous sublinear functions or equivalently if there exists a a pair
$\mathscr{D} f(x) = [\underline{\partial} f(x), \overline{\partial} f(x)]$ of convex weak$^*$ compact sets
$\underline{\partial} f(x), \overline{\partial} f(x) \subset X^*$ such that
\begin{equation} \label{DefOfQuasidiff}
  f'(x, h) = \max_{v^* \in \underline{\partial} f(x)} \langle v^*, h \rangle +
  \min_{w^* \in \overline{\partial} f(x)} \langle w^*, h \rangle \quad \forall h \in X.
\end{equation}
The pair $\mathscr{D} f(x)$ is called a Dini (Hadamard) \textit{quasidifferential} of $f$ at $x$, while the sets
$\underline{\partial} f(x)$ and $\overline{\partial} f(x)$ are called the Dini (Hadamard) \textit{subdifferential} and
\textit{superdifferential} of $f$ at $x$ respectively.
\end{definition}

\begin{remark}
Following the usual convention we identify $X^*$ with $X$ in the case when $X$ is either a finite dimensional or a
Hilbert space. Therefore, in particular, if $X = \mathbb{R}^n$, then a quasidifferential is a pair of convex compact
subsets of $\mathbb{R}^n$, while if $X$ is a Hilbert space, then a quasidifferential is a pair of weakly compact convex
subsets of $X$.
\end{remark}

A calculus of quasidifferentiable functions can be found in \cite{DemRub_book}. Here we only mention that any finite
DC (difference-of-convex) function is Hadamard quasidifferentiable. Note also that a quasidifferential of a function
$f$ is not unique, since for any quasidifferential $\mathscr{D} f(x)$ of $f$ at $x$ and any weak$^*$ compact
convex set $C \subset X^*$ the pair $[ \underline{\partial} f(x) + C, \overline{\partial} f(x) - C ]$ is a
quasidifferential of $f$ at $x$ as well. 

In the general case quasidifferential mapping $\mathscr{D} f(\cdot)$ might not possess any continuity properties;
however, for many nonsmooth functions appearing in applications it is \textit{outer semicontinuous} (o.s.c.). Recall
that if a function $f$ is quasidifferentiable in a neighbourhood $U$ of a point $x \in X$, then a quasidifferential
mapping $\mathscr{D} f(\cdot)$ defined in this neighbourhood is said to be o.s.c. at $x$, if the corresponding
multifunctions $\underline{\partial} f \colon U \to X^*$ and $\overline{\partial} f \colon U \to X^*$ are o.s.c. at $x$,
i.e. for any open sets $V_1, V_2 \subset X^*$ such that $\underline{\partial} f(x) \subset V_1$ and
$\overline{\partial} f(x) \subset V_2$ there exists $\delta > 0$ such that $\underline{\partial} f(x') \subset V_1$ and
$\overline{\partial} f(x') \subset V_2$ for all $x' \in U$ with $\| x' - x \| < \delta$. As was pointed out in
\cite{Kuntz}, a quasidifferential of a continuously codifferentiable function is outer semicontinuous
(see~\cite{DemRub_book} for the definition of continuously codifferentiable function). Hence, in particular, the class
of functions for which there exists an o.s.c. quasidifferential mapping is closed under all standard algebraic
operations, the pointwise maximum and minimum of finite families of functions, and the composition with smooth
functions, since the class of continuously codifferentiable functions is closed under all these operations
\cite{DemRub_book,Dolgopolik_AbstrConvApprox,Dolgopolik_CodiffDescent}. Furthermore, any DC function has an o.s.c.
quasidifferential mapping. Indeed, if $f = f_1 - f_2$, where $f_1$ and $f_2$ are finite closed convex functions, then
one can define $\mathscr{D} f(\cdot) = [ \partial f_1(\cdot), - \partial f_2(\cdot) ]$, where $\partial f_i(\cdot)$ is
the subdifferential of $f_i$ in the sense of convex analysis. Note that this quasidifferential is correctly defined and
o.s.c. due to the fact that the subdifferential of a finite closed convex function defined on a Banach space is nonempty
at every point (see, e.g. \cite[Proposition~I.5.2. and Corollary~I.2.5]{EkelandTemam}) and outer semicontinuous.

Let us also recall a certain extension of the definition of quasidifferentiability to the case of vector-valued
functions that was utilized in \cite{Glover92,Uderzo2}.

\begin{definition}
Let $Y$ be a real Banach space, and $U \subset X$ be an open set. A function $F \colon U \to Y$ is called
\textit{scalarly quasidifferentible} at a point $x \in U$, if $F$ is Dini directionally differentiable at $x$, i.e. for
any $h \in X$ there exists the limit
$$
  F'(x, h) = \lim_{\alpha \to +0} \frac{1}{\alpha} \big( F(x + \alpha h) - F(x) \big),
$$
and for any $y^* \in Y^*$ the function $\langle y^*, F'(x, \cdot) \rangle$ can be represented as the difference of
continuous sublinear functions, i.e. there exists a pair of convex weak$^*$ compact sets 
$\underline{\partial} F(x; y^*), \overline{\partial} F(x; y^*) \subset X^*$ such that
$$
  \langle y^*, F'(x, h) \rangle = \max_{v^* \in \underline{\partial} F(x; y^*)} \langle v^*, h \rangle +
  \min_{w^* \in \overline{\partial} F(x; y^*)} \langle w^*, h \rangle   \quad \forall v \in X.
$$
For any $y^* \in Y^*$ the pair 
$\mathscr{D} F(x; y^*) = [\underline{\partial} F(x; y^*), \overline{\partial} F(x; y^*)]$ is called 
a \textit{scalar quasidifferential} of $F$ at $x$ (corresponding to $y^*$).
\end{definition}

\begin{remark}
Below, as usual, we use the term ``quasidifferential'', instead of ``Dini quasidifferential''. Also, when we say
that a function $f$ is quasidifferentiable at a point $x$, we suppose that a quasidifferential of $f$ at $x$ is given.
Alternatively, one can define a quasidifferential as an equivalence class and work with these equivalence classes;
however, in author's opinion this approach leads to somewhat cumbersome formulations of the main results. That
is why we do not adopt it in this article.
\end{remark}

\section{Metric Regularity of Quasidifferentiable Mappings}
\label{Section_MetricRegularity}

In this section we obtain necessary and/or sufficient conditions for the metric regularity of multifunctions in terms of
quasidifferentials. We also introduce an MFCQ-type constraint qualifications for parametric systems of
quasidifferentiable equalities and inequalities and prove that it ensures the metric regularity of a multifunction
associated with this system.

\subsection{General Conditions for Metric Regularity}

Let $(Y, d)$ be a complete metric space, and $F \colon X \rightrightarrows Y$ be a set-valued mapping with closed
values, whose graph is denoted by $\Gr F$. For any $y \in Y$, $r > 0$ and any set $C \subset Y$ denote
$B(y, r) = \{ z \in Y \mid d(y, z) \le r \}$ and $d(y, C) = \inf_{z \in C} d(y, z)$. As usual, we put 
$d(y, \emptyset) = + \infty$.

Recall that $F$ is called \textit{metrically regular} near a point $(\overline{x}, \overline{y}) \in \Gr F$, if there
exist $K > 0$ and $r > 0$ such that
\begin{equation} \label{MetricRegDef}
  d(x, F^{-1}(y)) \le K d(y, F(x)) \quad \forall (x, y) \in B(\overline{x}, r) \times B(\overline{y}, r).
\end{equation}
The greatest lower bound of all $K$ for which the inequality above is satisfied with some $r > 0$ is called the
\textit{norm of metric regularity} of $F$ near $(\overline{x}, \overline{y})$. For the general theory of metric
regularity see \cite{Ioffe,Ioffe_book,Aze}.

At first, our aim is to obtain sufficient conditions for the metric regularity of the set-valued mapping $F$ in the case
when the distance function $x \to d(y, F(x))$ is quasidifferentiable for any $(x, y)$ in a neighbourhood of 
$(\overline{x}, \overline{y})$. To this end, for any $y \in Y$ and $x \in X$ denote $\psi_y(x) = d(y, F(x))$, and define
\begin{equation} \label{StrongSlopeDef}
  |\nabla \psi_y|(x) = 
  \limsup_{u \to x, \psi_y(u) \to \psi_y(x)} \frac{\max\{ \psi_y(x) - \psi_y(u), 0 \}}{\| x - u \|}.
\end{equation}
The quantity $|\nabla \psi_y|(x)$ is called the \textit{strong slope} of $\psi_y$ at $x$.

Recall that under some natural assumptions on the functions $\psi_y(\cdot)$ the validity of the inequality 
$|\nabla \psi_y|(x) > K^{-1}$ for any $(x, y) \notin \Gr F$ in a neighbourhood of $(\overline{x}, \overline{y})$ is
\textit{sufficient} for the metric regularity of $F$ near $(\overline{x}, \overline{y})$ with the norm of metric
regularity no exceeding $K$. In the case when $Y$ is a Banach space, the validity of the inequality 
$|\nabla \psi_y|(x) \ge t^{-1}$ for any such $(x, y)$ and for all $t > K$ is also \textit{necessary} for the metric
regularity of $F$ near $(\overline{x}, \overline{y})$ with the norm of metric regularity no exceeding $K$ (see, e.g.
\cite[Theorem~2b]{Ioffe}). 

In the following theorem we demonstrate how the verification of the inequality $|\nabla \psi_y|(x) > K^{-1}$ (and, thus,
the metric regularity of the multifunction $F$) can be significantly simplified in the case when the distance functions 
$\psi_y(\cdot) = d(y, F(\cdot))$ are quasidifferentiable.

\begin{theorem} \label{Thrm_MetricRegularity_General}
Let for any $y \in Y$ the function $\psi_y(\cdot)$ be lower semicontinuous (l.s.c.), and 
let $(\overline{x}, \overline{y}) \in \Gr F$ and $K > 0$ be given. Suppose that there exists $r > 0$ such that for any 
$(x, y) \in B(\overline{x}, r) \times B(\overline{y}, r)$ with $y \notin F(x)$ the function $\psi_y(\cdot)$ is
quasidifferentiable at $x$, and there exists $w^* \in \overline{\partial} \psi_y(x)$ for which
\begin{equation} \label{DistSubdiffQuasiDiff}
  d\big( \mathbb{O}, \underline{\partial} \psi_y(x) + w^* \big) > \frac{1}{K}.
\end{equation}
Then for any $(x, y) \in B(\overline{x}, r) \times B(\overline{y}, r)$ such that 
$K d(y, F(x)) < r - d(x, \overline{x})$ one has $d(x, F^{-1}(y)) \le K d(y, F(x))$, which, in particular, implies that
the set-valued mapping $F$ is metrically regular near $(\overline{x}, \overline{y})$ with the norm of metric regularity
not exceeding $K$.

Moreover, suppose that $Y$ is a Banach space, $X$ is finite dimensional, and for any $y \in Y$ the functions
$\psi_y(\cdot)$ are Hadamard quasidifferentiable on $B(\overline{x}, r)$ with some $r > 0$. Then for 
the metric regularity of $F$ near $(\overline{x}, \overline{y})$ with the norm of metric regularity not exceeding $K$ it
is necessary and sufficient that for any $t > K$ there exists a neighbourhood $U$ of $(\overline{x}, \overline{y})$ such
that for any $(x, y) \in U \setminus \Gr F$ there exists $w^* \in \overline{\partial} \psi_y(x)$ for which 
$d(\mathbb{O}, \underline{\partial} \psi_y(x) + w^*) \ge t^{-1}$.
\end{theorem}

\begin{proof}
Let us show that under the assumptions of the theorem one has
$$
  | \nabla \psi_y |(x) > K^{-1} \quad 
  \forall (x, y) \in \Big( B(\overline{x}, r) \times B(\overline{y}, r) \Big) \setminus \Gr F.
$$
Then applying \cite[Theorem~2b]{Ioffe} one obtains the desired result.

Indeed, fix $(x, y) \in B(\overline{x}, r) \times B(\overline{y}, r)$ with $y \notin F(x)$. From
\eqref{DistSubdiffQuasiDiff} it follows that for some $\varepsilon > 0$ the convex compact subsets 
$B(\mathbb{O}, K^{-1} + \varepsilon)$ and  $\underline{\partial} \psi_y(x) + w^*$ of the space $X^*$ endowed with the
weak$^*$ topology are disjoint. Applying the separation theorem one obtains that there exists $h \in X$ with 
$\| h \| = 1$ such that
$$
  \langle v^*, h \rangle \le \langle x^*, h \rangle \quad \forall v^* \in \underline{\partial} \psi_y(x) + w^*
  \quad \forall x^* \in B(\mathbb{O}, K^{-1} + \varepsilon)
$$
or equivalently $\langle v^*, h \rangle \le - K^{-1} - \varepsilon < - K^{-1}$ for any 
$v^* \in \underline{\partial} \psi_y(x) + \{ w^* \}$. Hence with the use of the definition of quasidifferential
(see~\eqref{DefOfQuasidiff}) it is easy to check that $\psi_y'(x, h) < K^{-1}$. Therefore there exists a sequence 
$\{ \alpha_n \} \subset (0, + \infty)$ such that $\alpha_n \to 0$ as $n \to \infty$ and
$$
  \lim_{n \to \infty} \frac{\psi_y(x + \alpha_n h) - \psi_y(x)}{\alpha_n} < - \frac{1}{K}.
$$
Consequently, $\psi_y(x) - \psi_y(x + \alpha_n h) > 0$ for any sufficiently large $n \in \mathbb{N}$, and
$$
  \limsup_{n \to \infty} \frac{\max\{ \psi_y(x) - \psi_y(x + \alpha_n h), 0 \}}{\alpha_n} > \frac{1}{K},
$$
which due to \eqref{StrongSlopeDef} yields $|\nabla \psi_y|(x) > K^{-1}$, since $\| h \| = 1$.

Let us now prove the second part of the theorem. Indeed, by \cite[Theorem~2b]{Ioffe} the multifunction $F$ is metrically
regular near $(\overline{x}, \overline{y})$ with the norm of metric regularity not exceeding $K$ iff for any $t > K$
there exists a neighbourhood $U$ of $(\overline{x}, \overline{y})$ such that $| \nabla \psi_y |(x) \ge t^{-1}$ for any 
$(x, y) \in U \setminus \Gr F$.

Taking into account the facts that $X$ is finite dimensional and the functions $x \to \psi_y(x)$ are Hadamard
quasidifferentiable, and applying \cite[Proposition~2.8]{Aze} one obtains that
$| \nabla \psi_y |(x) = - \min_{ \| h \| = 1} \psi_y'(x, h)$. Hence with the use of the explicit expression for the rate
of steepest descent of a quasidifferentiable function (see~\cite[Section~V.3.1]{DemRub_book}) one gets
$$
  | \nabla \psi_y |(x) = \max_{w^* \in \overline{\partial} \psi_y(x)} 
  \min_{v^* \in \underline{\partial} \psi_y(x) + \{ w^* \}} \| v^* \|,
  \quad \text{if } | \nabla \psi_y |(x) > 0.
$$
Consequently, $| \nabla \psi_y |(x) \ge t^{-1}$ iff $d(\mathbb{O}, \underline{\partial} \psi_y(x) + w^*) \ge t^{-1}$ for
some $w^* \in \overline{\partial} \psi_y(x)$, which implies the required result.	 
\end{proof}

\begin{remark} 
Taking into account the definition of quasidifferential \eqref{DefOfQuasidiff} it is easy to check that
condition~\eqref{DistSubdiffQuasiDiff} is satisfied for some $w^* \in \overline{\partial} \psi_y(x)$ iff there exists 
$h \in X$ with $\| h \| = 1$ such that $\psi_y'(x, h) < - K^{-1}$. Therefore, condition~\eqref{DistSubdiffQuasiDiff} is
invariant with respect to the choice of quasidifferentials of the functions $\psi_y$, since the directional derivative
$\psi_y'(x, \cdot)$ obviously does not depend on the choice of quasidifferential.
\end{remark}

\begin{remark} \label{Remark_ComparisonWithUderzo}
Sufficient conditions for the metric regularity of a continuous single-valued mapping $F$ between Banach
spaces in terms of quasidifferentials of the functions $\psi_y(x) = \| y - F(x) \|$ were first obtained by Uderzo
\cite{Uderzo} (see also \cite{Uderzo2}). However, the conditions in \cite{Uderzo} are more restrictive then the ones
stated in the theorem above. Indeed, by \cite[Theorem~4.3]{Uderzo} for the metric regularity of $F$ near a point
$(\overline{x}, F(\overline{x}))$ it is sufficient that there exist $m > 0$ and $r > 0$ such that for any 
$x \in B(\overline{x}, r)$ and $y \in B(F(\overline{x}), r)$ with $y \ne F(x)$ one has
\begin{equation} \label{Uderzo_Cond}
  d(\mathbb{O}, \underline{\partial} \psi_y(x) + w^*) > m \quad
  \forall w^* \in \overline{\partial} \psi_y(x).
\end{equation}
It is easy to see that this condition fails to hold true even for the very simple function 
$F(x_1, x_2) = |x_1| - |x_2|$, when $\overline{x} = \mathbf{0}_2$ and $\overline{y} = 0$ 
(here $X = \mathbb{R}^2$, $Y = \mathbb{R}$, and $\mathbf{0}_n$ is the zero vector from $\mathbb{R}^n$). Indeed, for 
$x = \mathbf{0}_2$ and any $y > 0$ a quasidifferential of the function $\psi_y(x) = |y - F(x)|$ has the form
$$
  \underline{\partial} \psi_y(\mathbf{0}_2) = \co\left\{ \begin{pmatrix} 0 \\ 1 \end{pmatrix}, 
  \begin{pmatrix} 0 \\ -1 \end{pmatrix} \right\}, \quad
  \overline{\partial} \psi_y(\mathbf{0}_2) = \co\left\{ \begin{pmatrix} 1 \\ 0 \end{pmatrix}, 
  \begin{pmatrix} -1 \\ 0 \end{pmatrix} \right\},
$$
and for $w^* = \mathbf{0}_2 \in \overline{\partial} \psi_y(\mathbf{0}_2)$ one has 
$\mathbf{0}_2 \in \underline{\partial} \psi_y(\mathbf{0}_2) + w^*$. Thus, condition \eqref{Uderzo_Cond} is not
satisfied. On the other hand, one can check that sufficient conditions from Theorem~\ref{Thrm_MetricRegularity_General}
are satisfied. Indeed, fix any $x \in \mathbb{R}^2$ and $y \in \mathbb{R}$ such that $y \ne F(x)$. Applying standard
rules of quasidifferential calculus \cite[Section~III.2]{DemRub_book} to the function $\psi_y(x) = |y - |x_1| + |x_2||$
one gets that
$$
  \underline{\partial} \psi_y(x) = \left\{ \begin{pmatrix} 0 \\ \svsign(x_2) \end{pmatrix} \right\}, \quad
  \overline{\partial} \psi_y(x) = \left\{ \begin{pmatrix} - \svsign(x_1) \\ 0 \end{pmatrix} \right\}
$$
in the case $y > |x_1| - |x_2|$, and
$$
  \underline{\partial} \psi_y(x) = \left\{ \begin{pmatrix} \svsign(x_1) \\ 0 \end{pmatrix} \right\}, \quad
  \overline{\partial} \psi_y(x) = \left\{ \begin{pmatrix} 0 \\ - \svsign(x_2) \end{pmatrix} \right\}
$$
in the case $y < |x_1| - |x_2|$. Here $\svsign(t) = \sign(t)$, if $t \ne 0$, and $\svsign(0) = [-1, 1]$. Let the space
$X = \mathbb{R}^2$ be equipped with the Euclidean norm. Then we have the following two cases:
\begin{enumerate}
\item{if $y > F(x)$, then setting $w^* = (-\sign(x_1), 0)^T \in \overline{\partial} \psi_y(x)$ in the case
$x_1 \ne 0$ and $w^* = (1, 0)^T \in \overline{\partial} \psi_y(x)$ in the case $x_1 = 0$ one gets that
$d( \mathbf{0}_2, \underline{\partial} \psi_y(x) + w^*) = \sqrt{2}$, provided $x_2 \ne 0$, and
$d( \mathbf{0}_2, \underline{\partial} \psi_y(x) + w^*) = 1$, if $x_2 = 0$;
}

\item{if $y < F(x)$, then setting $w^* = (0, -\sign(x_2))^T \in \overline{\partial} \psi_y(x)$ in the case $x_2 \ne 0$
and $w^* = (0, 1) \in \overline{\partial} \psi_y(x)$ in the case $x_2 = 0$ one obtains that 
$d( \mathbf{0}_2, \underline{\partial} \psi_y(x) + w^*) = \sqrt{2}$, if $x_1 \ne 0$, and
$d( \mathbf{0}_2, \underline{\partial} \psi_y(x) + w^*) = 1$, if $x_1 = 0$.
}
\end{enumerate}
Thus, condition \eqref{DistSubdiffQuasiDiff} is satisfied with any $K > 1$, and the function $F(x) = |x_1| - |x_2|$ is
metrically regular near the point $(\mathbf{0}_2, 0)$ by Theorem~\ref{Thrm_MetricRegularity_General}. 

Note also that condition~\eqref{Uderzo_Cond}, unlike \eqref{DistSubdiffQuasiDiff}, depends on the choice of
quasidifferential. For instance, it is not valid for the identity function $F(x) = x$, which is metrically regular near
any point (here $X = Y = \mathbb{R}$), if one chooses 
the pair $\underline{\partial} \psi_y(x) = - \sign(y - x) + [-1, 1]$ and $\overline{\partial} \psi_y(x) = [-1, 1]$, as a
quasidifferential of the function $\psi_y(x) = |y - F(x)| = |y - x|$ at every point $x$ such that $y \ne x$.
\end{remark}

Let us give another simple example illustrating Theorem~\ref{Thrm_MetricRegularity_General}.

\begin{example}
Let $X = Y = \mathbb{R}$ and $F(x) = \min\{ x, \max\{ x^3, 0 \} \}$, i.e. $F$ is single-valued. Let us check whether
this function is metrically regular near the point $(0, 0)$ with the use of Theorem~\ref{Thrm_MetricRegularity_General}.
By definition one has $\psi_y(x) = |y - \min\{ x, \max\{ x^3, 0 \} \}|$. The function $\psi_y(\cdot)$ is
quasidifferentiable and locally Lipschitz continuous, which implies that it is Hadamard quasidifferentiable. Applying
standard rules of quasidifferential calculus \cite[Section~III.2]{DemRub_book} one gets that
\begin{gather*}
  \underline{\partial} \psi_y(x) = \begin{cases}
    \{ - 1 \}, & \text{if } x \notin [0, 1], \\
    \{ - 3 x^2 \}, & \text{if } x \in (0, 1), \\
    [-3, -1], & \text{if } x = 1, \\
    [-1, 0], & \text{if } x = 0
  \end{cases}
  \quad \overline{\partial} \psi_y(x) = \{ 0 \} \quad \text{in the case } y > F(x), \\
  \underline{\partial} \psi_y(x) = \{ 0 \}, \quad
  \overline{\partial} \psi_y(x) = \begin{cases}
    \{ 1 \}, & \text{if } x \notin [0, 1], \\
    \{ 3 x^2 \}, & \text{if } x \in (0, 1), \\
    [1, 3], & \text{if } x = 1, \\
    [0, 1], & \text{if } x = 0
  \end{cases}
  \quad \text{in the case } y < F(x).
\end{gather*}
Therefore for $y = 0$ and any $x \in (0, 1)$ (note that in this case $y \ne F(x)$) one has 
$d(0, \underline{\partial} \psi_y(x) + w^*) = 3 x^2$ for any $w^* \in \overline{\partial} \psi_y(x)$. Choosing
sufficiently small $x > 0$ one obtains that $d(0, \underline{\partial} \psi_0(x) + w^*) < t^{-1}$ for any prespecified
$t > 0$ and for all $w^* \in \overline{\partial} \psi_0(x)$. Thus, by the second part of
Theorem~\ref{Thrm_MetricRegularity_General} one can conclude that $F$ is not metrically regular near $(0, 0)$. Let us
also verify this directly. Indeed, it is easily seen that $F^{-1}(y) = y$, if $y \notin [0, 1]$, and 
$F^{-1}(y) = y^{1/3}$, if $y \in [0, 1]$. Applying the definition of metric regularity \eqref{MetricRegDef} with $x = 0$
one gets that for the function $F$ to be metrically regular near $(0, 0)$ it is necessary that there exists $K > 0$
such that
$$
  d(0, F^{-1}(y)) = y^{1/3} \le K y = d(y, F(0))
$$
for any sufficiently small $y > 0$, which is obviously impossible. 
\end{example}

\subsection{Parametric Systems of Equalities and Inequalities}

In order to verify the metric regularity of a multifunction with the use of Theorem~\ref{Thrm_MetricRegularity_General},
one must check that  condition~(\ref{DistSubdiffQuasiDiff}) holds true at every point in a neighbourhood of a given
point $(\overline{x}, \overline{y})$, which is a common drawback of general results on metric regularity
(cf.~\cite{Ioffe,Aze}). However, as in the case of sufficient conditions in terms of various subdifferentials and
coderivatives, in some particular cases one can obtain sufficient conditions for the metric regularity that involve only
quasidifferentials of certain functions at the point $(\overline{x}, \overline{y})$ itself. Our next goal is to obtain
such conditions for a set-valued mapping associated with a parametric system of nonlinear equality and inequality
constraints.

Let $Y$ be a real Banach space, $P$ be a metric space of parameters, and let also $F \colon X \times P \to Y$ and 
$g_i \colon X \times P \to \mathbb{R}$, $i \in I = \{ 1, \ldots, m \}$, be given functions. For any $y \in Y$ and 
$z_i \in \mathbb{R}$, $i \in I$, consider the following parametric system
\begin{equation} \label{NonlinearSystem}
  F(x, p) = y, \quad g_i(x, p) \le z_i \quad i \in I.
\end{equation}
Denote by $\mathcal{S}(p, y, z) = \{ x \in X \mid F(x, p) = y, \: g_i(x, p) \le z_i, \: i \in I \}$ the solution set of
this system, where $z = (z_1, \ldots, z_m)^T \in \mathbb{R}^m$. We also denote 
$\mathcal{S}(p) = \mathcal{S}(p, \mathbb{O}_Y, \mathbf{0}_m)$, and sometimes use the notation $F_p(x) = F(x, p)$.

In the case when the functions $F(\cdot, p)$ and $g_i(\cdot, p)$ are continuously Fr\'{e}chet differentiable, 
the multifunction $\Phi_p(x) = \{ F(x, p) \} \times \prod_{i = 1}^m [g_i(x, p), + \infty)$ associated with system
\eqref{NonlinearSystem} is metrically regular near a given point if and only if the Mangasarian-Fromovitz constrain
qualification holds at this point, i.e. the Fr\'{e}chet derivative $D_x F(x, p)$ is a surjective mapping, and there
exists $h \in X$ such that $D_x F(x, p)[h] = \mathbb{O}$, while $D_x g_i(x, p)[h] < 0$ for any $i \in I$ such that 
$g_i(x, p) = z_i$ (see, e.g. \cite[Corollary~2.1]{Cominetti}). Our aim is to extend this results to the case when 
the functions $F(\cdot, p)$ and $g_i(\cdot, p)$ are only quasidifferentiable.

Being inspired by the results of \cite{Uderzo2}, let us introduce a constraint qualification in terms of
quasidifferentials that ensures the metric regularity of the multifunction associated with 
system \eqref{NonlinearSystem}. For the sake of shortness we consider the case $y = \mathbb{O}$ and $z = \mathbf{0}_m$
only, since the general case can be easily reduced to this one by replacing $F(x, p)$ with $F(x, p) - y$, and 
$g_i(x, p)$ with $g_i(x, p) - z_i$. Suppose that the functions $g_i(\cdot, \overline{p})$, $i \in I$, are
quasidifferentiable at a point $\overline{x}$ such that $\overline{x} \in \mathcal{S}(\overline{p})$, and the mapping
$F(\cdot, \overline{p})$ is scalarly quasidifferentiable at this point, and denote their quasidifferentials at this
point by $\mathscr{D}_x g_i(\overline{x}, \overline{p})$ and $\mathscr{D}_x F(\overline{x}, \overline{p}; y^*)$, 
$y^* \in Y^*$, respectively. Introduce the sets 
\begin{align*}
  [\mathscr{D}_x g_i(\overline{x}, \overline{p})]^+ 
  &= \underline{\partial}_x g_i(\overline{x}, \overline{p})
  + \overline{\partial}_x g_i(\overline{x}, \overline{p}), \\
  [\mathscr{D}_x F(\overline{x}, \overline{p}; y^*)]^+ 
  &= \underline{\partial}_x F(\overline{x}, \overline{p}; y^*) 
  + \overline{\partial}_x F(\overline{x}, \overline{p}; y^*).
\end{align*}
These sets are sometimes called \textit{quasidifferential sums}, and they were considered e.g. in \cite{Uderzo2}.
Note that quasidifferential sums are \textit{not} invariant with respect to the choice of the corresponding
quasidifferentials. For example, for the function $f(x) = |x|$ both $\mathscr{D}_1 f(0) = [[-1, 1], \{ 0 \}]$
and $\mathscr{D}_2 f(0) = [[-2, 2], [-1, 1]]$ are quasidifferentials of $f$ at $x$, and 
$[\mathscr{D}_1 f(0)]^+ = [ -1, 1] \ne [-3, 3] = [\mathscr{D}_2 f(0)]^+$. Thus, all conditions below are not invariant
with respect to the choise of quasidifferentials.

For any $x \in X$ and $p \in P$ define $I(x, p) = \{ i \in I \mid g_i(x, p) = 0 \}$, and 
denote $S_X = \{ x \in X \mid \| x \| = 1 \}$.

\begin{definition} \label{Def_qdMFCQ}
One says that \textit{the Mangasarian-Fromovitz constraint qualification in terms of quasidifferentials} (q.d.-MFCQ)
holds at $(\overline{x}, \overline{p})$, if 
\begin{equation} \label{HiddenLinearIndependence}
  \inf_{y^* \in S_{Y^*}} 
  \inf\{ \| v^* \| \colon v^* \in [\mathscr{D}_x F(\overline{x}, \overline{p}; y^*)]^+ \} > 0,
\end{equation}
and there exists $\overline{h} \in X$ such that $\langle v^*, \overline{h} \rangle = 0$ for all 
$v^* \in [\mathscr{D}_x F(\overline{x}, \overline{p}; y^*)]^+$ and $y^* \in Y^*$, while
$\langle v^*, \overline{h} \rangle < 0$ for all $v^* \in [\mathscr{D}_x g_i(\overline{x}, \overline{p})]^+$ and 
$i \in I(\overline{x}, \overline{p})$.
\end{definition}

Let us point out how q.d.-MFCQ is connected with the standard MFCQ. To this end, recall that nonempty subsets
$A_1, \ldots, A_s$ of a linear space $E$ are said to be \textit{linearly independent} (or to have \textit{full rank}),
if the inclusion $\mathbb{O} \in \lambda_1 A_1 + \ldots + \lambda_n A_n$ with $\lambda_i \in \mathbb{R}$ is valid only
for $\lambda_i = 0$, $i \in \{1, \ldots, s \}$. Clearly, the sets $A_i$, $i \in \{ 1, \ldots, s \}$ are linearly
independent iff for any $x_i \in A_i$, $i \in \{ 1, \ldots, s \}$, the vectors $x_1, \ldots, x_s$ are linearly
independent. 

\begin{proposition} \label{Prp_Equiv_qd_MFCQ}
Let $Y$ be the space $\mathbb{R}^l$ equipped with the Euclidean norm $|\cdot|$, and 
$F(\cdot) = (f_1(\cdot), \ldots, f_l(\cdot))^T$, where the functions $f_j \colon X \times P \to \mathbb{R}$ are
quasidifferentiable in $x$ at $(\overline{x}, \overline{p})$. Then the mapping $F(\cdot, \overline{p})$ is scalarly
quasidifferentiable at $\overline{x}$. Moreover, q.d.-MFCQ holds at $(\overline{x}, \overline{p})$ iff the sets 
$[\mathscr{D}_x f_j(\overline{x}, \overline{p})]^+$, $1 \le j \le l$, are linearly independent, and there exists 
$\overline{h} \in X$ such that $\langle v^*, \overline{h} \rangle = 0$ 
for all $v^* \in [\mathscr{D}_x f_j(\overline{x}, \overline{p})]^+$ and $1 \le j \le l$,
while $\langle v^*, \overline{h} \rangle < 0$ for all $v^* \in [\mathscr{D}_x g_i(\overline{x}, \overline{p})]^+$ and 
$i \in I(\overline{x}, \overline{p})$.
\end{proposition}

\begin{proof}
From the fact that the functions $f_j(\cdot, \overline{p})$ are quasidifferentiable at $\overline{x}$ it
follows that the mapping $F(\cdot, \overline{p})$ is directionally differentiable at this point, and
$$
  [F(\cdot, \overline{p})]'(\overline{x}, h) = \Big( [f_1(\cdot, \overline{p})]'(\overline{x}, h), \ldots,
  [f_l(\cdot, \overline{p})]'(\overline{x}, h) \Big)^T
$$
for any $h \in X$. Therefore, for any $y^* = (y_1, \ldots, y_l)^T \in \mathbb{R}^l$ one has
$$
  \langle y^*, [F(\cdot, \overline{p})]'(\overline{x}, h) \rangle
  = \sum_{j = 1}^l y_j 
  \Big( \max_{v^* \in \underline{\partial}_x f_j(\overline{x}, \overline{p})} \langle v^*, h \rangle
  + \min_{w^* \in \overline{\partial}_x f_j(\overline{x}, \overline{p})} \langle w^*, h \rangle \Big),
$$
which implies that $F(\cdot, \overline{p})$ is scalarly quasidifferentiable at $\overline{x}$, and for any $y^*$ one
can define
\begin{gather*}
  \underline{\partial}_x F(\overline{x}, \overline{p}; y^*)
  = \sum_{j = 1}^l \Big( [y_j]_+ \underline{\partial}_x f_j(\overline{x}, \overline{p})
  - [-y_j]_+ \overline{\partial}_x f_j(\overline{x}, \overline{p}) \Big), \\
  \overline{\partial}_x F(\overline{x}, \overline{p}; y^*) 
  = \sum_{j = 1}^l \Big( [y_j]_+ \overline{\partial}_x f_j(\overline{x}, \overline{p})
  - [-y_j]_+ \underline{\partial}_x f_j(\overline{x}, \overline{p}) \Big),
\end{gather*}
where $[t]_+ = \max\{ t, 0 \}$ for any $t \in \mathbb{R}$. Hence for any $y^*$ one has
\begin{equation} \label{SumQuasidiff}
  [\mathscr{D}_x F(\overline{x}, \overline{p}; y^*)]^+
  = \sum_{j = 1}^l y_j [\mathscr{D}_x f_j(\overline{x}, \overline{p})]^+.
\end{equation}
Consequently, if \eqref{HiddenLinearIndependence} holds true, then the sets 
$[\mathscr{D}_x f_j(\overline{x}, \overline{p})]^+$, $1 \le j \le l$, are linearly independent, since otherwise
$0 \in [\mathscr{D}_x F(\overline{x}, \overline{p}; y^*)]$ for $y^* = \lambda / |\lambda|$, where 
$\lambda \in \mathbb{R}^l$, $\lambda \ne \mathbf{0}_l$ is such that 
$\mathbb{O} \in \lambda_1 [\mathscr{D}_x f_1(\overline{x}, \overline{p})]^+ + \ldots
+ \lambda_l [\mathscr{D} f_l(\overline{x}, \overline{p})]^+$, which is impossible. Conversely, if the sets 
$[\mathscr{D}_x f_j(\overline{x}, \overline{p})]^+$, $1 \le j \le l$, are linearly independent, then
$\mathbb{O} \notin [\mathscr{D}_x F(\overline{x}, \overline{p}; y^*)]^+$ for any $y^* \ne \mathbf{0}_l$. Applying 
the separation theorem and the fact that the set $[\mathscr{D}_x F(\overline{x}, \overline{p}; y^*)]^+$ is weak${}^*$
compact one obtains that there exist $h \in X$ and $\delta > 0$ such that $\langle v^*, h \rangle \ge \delta$ for all 
$v^* \in [\mathscr{D}_x F(\overline{x}, \overline{p}; y^*)]^+$.
Therefore $\inf\{ \| v^* \| \mid v^* \in [\mathscr{D}_x F(\overline{x}, \overline{p}; y^*)]^+ \} > 0$ for any 
$y^* \ne \mathbf{0}_l$. Hence taking into account the facts that this infimum is obviously continuous with
respect to $y^*$ (see~\eqref{SumQuasidiff}), and the unit sphere in $\mathbb{R}^l$ is compact one gets that
\eqref{HiddenLinearIndependence} holds true. It remains to note that the equivalence between the second conditions from
q.d.-MFCQ and the proposition (the existence of $\overline{h}$) follows directly from \eqref{SumQuasidiff}.	 
\end{proof}

\begin{remark} \label{Remark_Equivalent_qdMFCQ}
With the use of the separation theorem one can easily check that under the assumptions of the proposition above the
vector $\overline{h}$ from q.d.-MFCQ exists iff 
\begin{equation} \label{qdMFCQ_geometric}
  \co\big\{ [\mathscr{D}_x g_i(\overline{x}, \overline{p})]^+ \mid i \in I(\overline{x}, \overline{p}) \big\} \cap 
  \cl\linhull\big\{ [\mathscr{D}_x f_j(\overline{x}, \overline{p})]^+ \mid 1 \le j \le l \big\} = \emptyset,
\end{equation}
where the closure is taken in the weak$^*$ topology. Furthermore, if $X$ is finite dimensional, then this span is
weak$^*$ closed, and \eqref{qdMFCQ_geometric} is equivalent to the following condition: for any 
$v^*_i \in [\mathscr{D}_x g_i(\overline{x}, \overline{p})]^+$, $i \in I(\overline{x}, \overline{p})$, and
$w_k^* \in \bigcup_{1 \le j \le l} [\mathscr{D}_x f_j(\overline{x}, \overline{p})]^+$, $1 \le k \le n$, where $n$ is the
dimension of $X$, there exists $\overline{h} \in X$ such that
\begin{equation} \label{qdMFCQ_pointwise}
  \langle v^*_i, \overline{h} \rangle < 0 \quad \forall i \in I(\overline{x}, \overline{p}), \quad
  \langle w^*_k, \overline{h} \rangle = 0 \quad \forall k \in \{ 1, \ldots, n \}.
\end{equation}
The implication $\eqref{qdMFCQ_geometric} \implies \eqref{qdMFCQ_pointwise}$ follows from the separation theorem, while
the opposite implication follows from the fact that if the intersection in \eqref{qdMFCQ_geometric} is not empty, then
it is impossible to find $\overline{h}$ satisfying \eqref{qdMFCQ_pointwise} for those $v^*_i$ and $w_k^*$ that
correspond to a vector from the intersection. Note that condition \eqref{qdMFCQ_pointwise} is, in a sense, a
``pointwise'' version of the second condition from q.d.-MFCQ. Let us finally point out that in the case when $l = 1$ the
``linear independence condition'' from q.d.-MFCQ is reduced to 
$\mathbb{O} \notin [\mathscr{D}_x f_1(\overline{x}, \overline{p})]^+$.
\end{remark}

Likewise the standard Mangasarian-Fromowitz constraint qualification, q.d.-MFCQ can be used to obtain sufficient
conditions for metric regularity. For the sake of simplicity we consider only the case when the functions $F$ and $g_i$
are continuous on $X \times P$, although the theorem below holds true under weaker assumptions. Note also that in 
the theorem below, unlike in the main results of \cite{Uderzo2}, we do not assume that the Banach space $Y$ admits a
Fr\'{e}chet smooth renorming.

\begin{theorem} \label{Thrm_ParamSystem_MetricReg}
Suppose that the functions $F$ and $g_i$, $i \in I$, are continuous. Let also 
a point $(\overline{x}, \overline{p}) \in X \times P$ be such that $\overline{x} \in \mathcal{S}(\overline{p})$, and
there exist a neighbourhood $U$ of $(\overline{x}, \overline{p})$ such that
\begin{enumerate}
\item{for any $(x, p) \in U$ the mapping $F(\cdot, p)$ is scalarly quasidifferentiable at $x$, and 
the functions $g_i(\cdot, p)$, $i \in I(\overline{x}, \overline{p})$, are quasidifferentiable at $x$;
}

\item{the multifunctions $\mathscr{D}_x g_i(\cdot)$, $i \in I(\overline{x}, \overline{p})$, are o.s.c. at 
$(\overline{x}, \overline{p})$, while the multifunction $(x, p) \mapsto [\mathscr{D}_x F(x, p; y^*)]^+$ is o.s.c. at
$(\overline{x}, \overline{p})$ uniformly with respect to $y^* \in S_{Y^*}$, i.e. for any $\varepsilon > 0$ there exists
$\delta > 0$ such that 
$[\mathscr{D}_x F(x, p; y^*)]^+ \subseteq [\mathscr{D}_x F(\overline{x}, \overline{p}; y^*)]^+ 
+ B(\mathbb{O}, \varepsilon)$ for 
all $y^* \in S_{Y^*}$ and $(x, p) \in B(\overline{x}, \delta) \times B(\overline{p}, \delta)$;
\label{Assumpt_OSC_QuasidiffMaps}
}

\item{the set $D(y) = \{ [\mathscr{D}_x F(\overline{x}, \overline{p}; y^*)]^+ 
\mid y^* \in S_{Y^*} \colon \langle y^*, y \rangle = \| y \| \}$ is weak${}^*$ closed and convex for any $y \in S_Y$.
\label{Assump_WeakStarClosedScalarQuasidiff}
}
\end{enumerate}
Suppose, finally, that q.d.-MFCQ holds at $(\overline{x}, \overline{p})$. Then there exist $K > 0$, a neighbourhood
$V$ of $(\overline{x}, \overline{p})$, and a neighbourhood $W$ of zero in $Y \times \mathbb{R}^m$ such that
\begin{equation} \label{MetricReg_SystemOfEq_Ineq}
  d( x, \mathcal{S}(p, y, z) ) 
  \le K \Big( \| F(x, p) - y \| + \sum_{i = 1}^m \max\{ g_i(x, p) - z_i, 0 \} \Big)
\end{equation}
for all $(x, p) \in V$ and $(y, z) \in W$. Therefore, in particular, the set-valued mapping 
$\Phi_p \colon X \rightrightarrows Y \times \mathbb{R}^m$, 
$\Phi_p(x) = \{ F(x, p) \} \times \prod_{i = 1}^m [g_i(x, p), + \infty)$ is metrically regular near the point
$(\overline{x}, (\mathbb{O}_Y, \mathbf{0}_m))$ with the norm of metric regularity not exceeding $K$ for all $p$ in a
neighbourhood of $\overline{p}$.
\end{theorem}

\begin{proof}
Let $\overline{r} > 0$ be such that $B(\overline{x}, \overline{r}) \times B(\overline{p}, \overline{r}) \subset U$. Our
aim is to prove that there exist $r \in (0, \overline{r})$ and $K > 0$ such that for any $p \in B(\overline{p}, r)$ one
has $|\nabla \psi_{(y, z, p)}|(x) > K^{-1}$ for all $(y, z) \in B((\mathbb{O}_Y, \mathbf{0}_m), r)$ and 
$x \in B(\overline{x}, r)$ such that $(y, z) \notin \Phi_p(x)$, where 
$\psi_{(y, z, p)}(x) = d( (y, z), \Phi_p(x) )$, and the space $Y \times \mathbb{R}^m$ is equipped with 
the norm $\| (y, z) \| = \| y \| + \sum_{i = 1}^m |z_i|$. Then applying \cite[Theorem~2b]{Ioffe} one obtains that 
$d(x, \Phi_p^{-1}(y, z)) \le K d( (y, z), \Phi_p(x) )$ for all $x \in B(\overline{x}, r)$, $p \in B(\overline{p}, r)$,
and $(y, z) \in B((\mathbb{O}_Y, \mathbf{0}_m), r)$ such that $K d((y, z), \Phi_p(x)) < r - \| x - \overline{x} \|$,
i.e. \eqref{MetricReg_SystemOfEq_Ineq} holds true for all such $x$, $p$, $y$, and $z$. With the use of the continuity of
the functions $F$ and $g_i$ and the fact that $\overline{x} \in \mathcal{S}(\overline{p})$, i.e. 
$(\mathbb{O}_Y, \mathbf{0}_m) \in \Phi_{\overline{p}}(\overline{x})$, one can find $\delta < r$ such that 
$K d((y, z), \Phi_p(x)) < r - \| x - \overline{x} \|$ for all $x \in B(\overline{x}, \delta)$, 
$p \in B(\overline{p}, \delta)$ and $(y, z) \in B((\mathbb{O}_Y, \mathbf{0}_m), \delta)$, which implies that
\eqref{MetricReg_SystemOfEq_Ineq} holds true for all such $x$, $p$, $y$, and $z$, and the proof is complete.

Before we proceed to the proof of the inequality $|\nabla \psi_{(y, z, p)}|(x) > K^{-1}$, let us first compute 
the directional derivative of the mapping $\| F(\cdot, p) - y \|$. Denote $\omega(y) = \| y \|$. Recall that 
$\partial \omega(y) = \{ y^* \in S_{Y^*} \mid \| y \| = \langle y^*, y \rangle \}$ for any $y \ne \mathbb{O}$, where 
$\partial \omega(y)$ is the subdifferential of $\omega$ at $y$ in the sense of convex analysis. Fix $(x, p) \in U$ and
$y \in Y$. From the definition of scalar quasidifferentiability it follows that for any $h \in X$ one has
$$
  F_p(x + \alpha h) - F_p(x) = \alpha F_p'(x, h) + o(\alpha) \quad \forall \alpha \ge 0,
$$
where $\| o(\alpha) \| / \alpha \to 0$ as $\alpha \to + 0$ (recall that $F_p(x) = F(x, p)$). Hence
\begin{multline*}
  \Big| \| F_p(x + \alpha h) - y \| - \| F_p(x) - y \| - \alpha \omega'\big( F_p(x) - y, F_p'(x, h) \big) \Big| \\
  = \Big| \| F_p(x) - y + \alpha F_p'(x, h) + o(\alpha) \| - \| F_p(x) - y \| 
  - \alpha \omega'\big( F_p(x) - y, F_p'(x, h) \big) \Big| \\
  \le \Big| \| F_p(x) - y + \alpha F_p'(x, h) \| - \| F_p(x) - y \| 
  - \alpha \omega'\big( F_p(x) - y, F_p'(x, h) \big) \Big| + \| o(\alpha) \|.
\end{multline*}
Dividing this inequality by $\alpha$ and passing to the limit as $\alpha \to + 0$ one gets that the function 
$\| F_p(\cdot) - y \|$ is directionally differentiable at $x$, and for any $h \in X$ and $y \in Y$ one has
\begin{multline} \label{NormEqualConstr_DirectDeriv}
  \| F_p(\cdot) - y \|'(x, h) = \omega'(F_p(x) - y, F_p'(x, h))
  = \sup_{y^* \in \partial \omega(F_p(x) - y)} \langle y^*, F_p'(x, h) \rangle \\
  = \sup_{y^* \in \partial \omega(F_p(x) - y)} 
  \Big( \max_{v^* \in \underline{\partial}_x F(x, p; y^*)} \langle v^*, h \rangle 
  + \min_{w^* \in \overline{\partial}_x F(x, p; y^*)} \langle w^*, h \rangle \Big) \\
  \le \sup_{y^* \in \partial \omega(F_p(x) - y)} \max_{v^* \in [\mathscr{D}_x F(x, p; y^*)]^+} \langle v^*, h \rangle,
\end{multline}
if $F(x, p) \ne y$, while
\begin{equation} \label{NormEqualConstr_DirectDeriv_AtZero}
  \| F_p(\cdot) - y) \|'(x, h) = \| F_p'(x, h) \| \le 
  \sup_{y^* \in S_{Y^*}} \max_{v^* \in [\mathscr{D}_x F(x, p; y^*)]^+} \langle v^*, h \rangle,
\end{equation}
in the case $F(x, p) = y$, since $\| y \| = \sup_{y^* \in S_{Y^*}} \langle y^*, y \rangle$.

Now we can utilize q.d.-MFCQ and the outer semicontinuity of the quasidifferential mappings to prove the inequality
$|\nabla \psi_{(y, z, p)}|(x) > K^{-1}$. Let $\varkappa > 0$ be any number smaller than 
the infimum in \eqref{HiddenLinearIndependence}. From assumption \ref{Assump_WeakStarClosedScalarQuasidiff}, the fact
that the set $D(y)$ is convex, and the separation theorem it follows that for any $y \in S_Y$ there exists $h_y$ with
$\| h_y \| = 1$ such that $\langle v^*, h_y \rangle \le - \varkappa$ for all $v^* \in D(y)$. With the use of the second
condition in q.d.-MFCQ one obtains that $\langle v^*, h_y + t \overline{h} \rangle \le - \varkappa$ for all 
$v^* \in D(y)$ and $t \ge 0$, where the vector $\overline{h}$ is from q.d.-MFCQ. Hence applying the fact that 
the mapping $(x, p) \mapsto [\mathscr{D}_x F(x, p; y^*)]^+$ is o.s.c. at $(\overline{x}, \overline{p})$ uniformly with
respect to $y^* \in S_{Y^*}$,  one gets that for any $t \ge 0$ there exists $r_1(t) \in (0, \overline{r})$ such that for
any $y \in S_Y$ one has
\begin{equation} \label{EqualConstr_DescentDirection}
  \langle v^*, h_y + t \overline{h} \rangle \le - \frac{\varkappa}{2} \quad 
  \forall v^* \in [\mathscr{D}_x F(x, p; y^*)]^+ \quad \forall y^* \in \partial \| \cdot \|(y)
\end{equation}
for all $(x, p) \in B(\overline{x}, r_1(t)) \times B(\overline{p}, r_1(t))$.  Furthermore, from the second condition in
q.d.-MFCQ and assumption \ref{Assumpt_OSC_QuasidiffMaps} it follows that for any $t \ge 0$ there exists 
$r_2(t) \in (0, \overline{r})$ such that
\begin{equation} \label{EqualConstr_SlowAscentAtOrigin}
  \langle v^*, t \overline{h} \rangle \le \frac{\varkappa}{4} \quad
  \forall v^* \in [\mathscr{D}_x F(x, p; y^*)]^+ \quad \forall y^* \in S_{y^*}
\end{equation}
for all $(x, p) \in B(\overline{x}, r_2(t)) \times B(\overline{p}, r_2(t))$. 

Applying the second condition in q.d.-MFCQ, and the facts that $\| h_y \| = 1$ for any $y \in S_Y$ and the sets
$[\mathscr{D}_x g_i(\overline{x}, \overline{p})]^+$ are obviously weak${}^*$ compact (and thus bounded) one can find
$t_0 > 0$ such that $\langle v^*, h_y + t_0 \overline{h} \rangle \le - \varkappa$ for 
all $v^* \in [\mathscr{D}_x g_i(\overline{x}, \overline{p})]^+$, $i \in I(\overline{x}, \overline{p})$, 
and $y \in S_{Y}$. Hence with the use of the outer semicontinuity of the mappings
$\mathscr{D}_x g_i(\cdot)$ at $(\overline{x}, \overline{p})$ one obtains that there exists $r_3 \in (0, \overline{r})$
such that
\begin{equation} \label{InequalConstr_DescentDirection}
  \langle v^*, h_y + t_0 \overline{h} \rangle \le - \frac{\varkappa}{2} \quad
  \forall v^* \in [\mathscr{D}_x g_i(x, p)]^+ \: \forall i \in I(\overline{x}, \overline{p})
  \: \forall y \in S_{Y}.
\end{equation}
for all $(x, p) \in B(\overline{x}, r_3) \times B(\overline{p}, r_3)$. Finally, since $g_i$ are continuous, there exists
$r_4 \in (0, \overline{r})$ and $\varepsilon > 0$ such that $g_i(x, p) < - \varepsilon$ for any 
$(x, p) \in B(\overline{x}, r_4) \times B(\overline{p}, r_4)$ and $i \notin I(\overline{x}, \overline{p})$.

Define $r = \min\{ r_1(t_0), r_2(t_0), r_3, r_4, \varepsilon / 2 \}$, and fix any 
$(x, p) \in B(\overline{x}, r) \times B(\overline{p}, r)$ and $(y, z) \in B((\mathbb{O}_Y, \mathbf{0}_m), r)$ such that 
$(y, z) \notin \Phi_p(x)$. Note that $g_i(x, p) - z_i < 0$ for any $i \notin I(\overline{x}, \overline{p})$, since 
$r \le \min\{ r_4, \varepsilon / 2 \}$, which implies that $g_i(\cdot) - z_i < 0$ in a neighbourhood of $(x, p)$ for
any such $i$. Hence
$$
  d( (y, z), \Phi_{q}(\xi) ) = \| F(\xi, q) - y \| + 
  \sum_{i \in I(\overline{x}, \overline{p})} \max\{ g_i(\xi, q) - z_i, 0 \}
$$
for any $(\xi, q)$ in a neighbourhood of $(x, p)$, i.e. the indices $i \notin I(\overline{x}, \overline{p})$ can be
discarded from consideration. Observe also that
\begin{equation} \label{InequalConstr_DirectionalDeriv}
  \max\{ g_i(\cdot, p) - z_i, 0 \}'(x, h) = \begin{cases}
    [g_i(\cdot, p)]'(x, h), & \text{if } g_i(x, p) > z_i, \\
    \max\{ [g_i(\cdot, p)]'(x, h), 0 \}, & \text{if } g_i(x, p) = z_i, \\
    0, & \text{if } g_i(x, p) < z_i,
  \end{cases}
\end{equation}
and $[g_i(\cdot, p)]'(x, h) \le \max_{v^* \in [\mathscr{D}_x g_i(x, p)]^+} \langle v^*, h \rangle$ for any $h \in X$.

If $F(x, p) \ne y$, then with the use of \eqref{NormEqualConstr_DirectDeriv}, \eqref{EqualConstr_DescentDirection},
\eqref{InequalConstr_DescentDirection}, and \eqref{InequalConstr_DirectionalDeriv} one obtains that
$$
  \psi_{(y, z, p)}'(x, \eta) = \| F(\cdot, p) - y \|'(x, \eta) 
  + \sum_{i \in I(\overline{x}, \overline{p})} \max\{ g_i(\cdot, p) - z_i, 0 \}'(x, \eta)
  \le - \frac{\varkappa}{2}
$$
where $\eta = h_w + t_0 \overline{h}$ and $w = (F(x, p) - y) / \| F(x, p) - y \|$ (here we used the fact that 
$\partial \| \cdot \| (F(x, p) - y) = \partial \| \cdot \| (w)$). Note that $\| \eta \| \le 1 + t_0 \| \overline{h} \|$,
since $\| h_w \| = 1$. 

On the other hand, if $F(x, p) = y$,
then there exists $k \in I(\overline{x}, \overline{p})$ such that $g_k(x, p) > z_k$. Consequently, applying
\eqref{NormEqualConstr_DirectDeriv_AtZero}, \eqref{EqualConstr_SlowAscentAtOrigin},
\eqref{InequalConstr_DescentDirection}, and \eqref{InequalConstr_DirectionalDeriv} one gets that
\begin{align*}
  \psi_{(y, z, p)}'(x, \eta) &=
  \| F(\cdot, p) - y \|'(x, \eta) + \max\{ g_k(\cdot, p) - z_k, 0 \}'(x, \eta) \\
  &+ \sum_{i \in I(\overline{x}, \overline{p}) \setminus \{ k \}} \max\{ g_i(\cdot, p) - z_i, 0 \}'(x, \eta)
  \le \frac{\varkappa}{4} - \frac{\varkappa}{2} = - \frac{\varkappa}{4},
\end{align*}
where $\eta = t_0 \overline{h}$. Thus, for any $(x, p) \in B(\overline{x}, r) \times B(\overline{p}, r)$ and 
$(y, z) \in B((\mathbb{O}_Y, \mathbf{0}_m), r)$ such that $(y, z) \notin \Phi_p(x)$ one has
$$
  |\nabla \psi_{(y, z, p)}|(x) \ge - \psi_{(y, z, p)}'\left(x, \frac{\eta}{\| \eta \|} \right)
  \ge \frac{\varkappa}{4(1 + t_0 \| \overline{h} \|)},
$$
and the proof is complete.	 
\end{proof}

\begin{remark} \label{Remark_MatrixQuasidiff}
Let $F$ be as in Proposition~\ref{Prp_Equiv_qd_MFCQ} and $X = \mathbb{R}^n$. In this case one can
reformulate the sufficient conditions for the metric regularity of the mapping $F$ from
the theorem above in a different way. Namely, let the set $\underline{\partial}_x F(\overline{x}, \overline{p})$
consists of all $l \times n$ matrices whose $j$-th row is a vector 
from $\underline{\partial}_x f_j(\overline{x}, \overline{p})$. The set 
$\overline{\partial}_x F(\overline{x}, \overline{p})$ is defined in a similar way. Then the pair 
$\mathscr{D}_x F(\overline{x}, \overline{p}) 
= [\underline{\partial}_x F(\overline{x}, \overline{p}), \overline{\partial}_x F(\overline{x}, \overline{p})]$ is, in
fact, a quasi\-differential of the mapping $F(\cdot, \overline{p})$ at $\overline{x}$
(see~\cite[Appendix~III]{DemRub_book}). From Theorem~\ref{Thrm_ParamSystem_MetricReg} it follows that
for the mapping $F(\cdot, p)$ to be metrically regular near $(\overline{x}, F(\overline{x}, p))$ with the norm of metric
regularity not exceeding some $K > 0$ for all $p$ in a neighbourhood of $\overline{p}$ it is sufficient that $l \le n$,
and all matrices from the set $[\mathscr{D}_x F(\overline{x}, \overline{p})]^+ = 
\underline{\partial}_x F(\overline{x}, \overline{p}) + \overline{\partial}_x F(\overline{x}, \overline{p})$ have full
rank. Note that a similar condition on the set $[\mathscr{D}_x F(\overline{x}, \overline{p})]^+$ was introduced by
Demyanov in \cite{Demyanov_NewtMeth} for the analysis of nonsmooth implicit functions and a nonsmooth Newton method for
codifferentiable vector-valued functions.
\end{remark}

\begin{remark}
It should be noted that in the case when $X = \mathbb{R}^n$ and $Y = \mathbb{R}^l$, 
Theorem~\ref{Thrm_ParamSystem_MetricReg} is, in essence, reduced to the sufficient conditions for metric regularity in
terms of the Clarke subdifferential \cite{Auslender,Borwein}. Indeed, if a function  $f \colon X \to \mathbb{R}$ is
quasidifferentiable at a point $x$, then, as it easy to see,
$$
  \min_{v^* \in [\mathscr{D} f(x)]^+} \langle v^*, h \rangle \le f'(x, h) 
  \le \max_{v^* \in [\mathscr{D} f(x)]^+} \langle v^*, h \rangle
  \quad \forall h \in X,
$$
i.e. the quasidifferential sum $[\mathscr{D} f(x)]^+$ is a \textit{convexificator} of $f$ at $x$ (see
\cite{DemyanovJeyakumar,JeyakumarLuc,Demyanov2000}). With the use of the separation theorem and the inequalities above
one can easily check that if $f$ is G\^{a}teaux differentiable at $x$, then $f'(x) \in [\mathscr{D} f(x)]^+$ regardless
of the choice of quasidifferential. Consequently, if $X = \mathbb{R}^n$, $f$ is Lipschitz continuous and
quasidifferentiable near $x$, and a quasidifferential mapping $\mathscr{D} f$ is o.s.c. at $x$, then 
$\partial_{Cl} f(x) \subseteq [\mathscr{D} f(x)]^+$, where $\partial_{Cl} f(x)$ is the Clarke subdifferential of $f$ at
$x$ \cite{Clarke}. 

With the use of \cite[Corollary~2]{Dolgopolik_CodiffDescent} one can verify that under the assumptions of
Theorem~\ref{Thrm_ParamSystem_MetricReg} the functions $F(\cdot, p)$ and $g_i(\cdot, p)$ are Lipschitz continuous near
$\overline{x}$ with the same Lipschitz constant for all $p$ in a neighbourhood of $\overline{p}$, provided $F$ has the
same form as in Proposition~\ref{Prp_Equiv_qd_MFCQ}. Therefore, if $X = \mathbb{R}^n$, then 
$\partial_{Cl} g_i(\cdot, \overline{p})(\overline{x}) \subseteq [\mathscr{D}_x g_i(\overline{x}, \overline{p})]^+$, and
the same inclusion holds true for $f_j(x, p)$. Thus, if $X = \mathbb{R}^n$ and $Y = \mathbb{R}^l$, then
Theorem~\ref{Thrm_ParamSystem_MetricReg} is a corollary to the sufficient conditions for metric regularity in terms of
the Clarke subdifferential \cite[Theorem~1.1]{Auslender} (see also \cite{Borwein}). On the other hand, if either $X$ or
$Y$ is infinite dimensional, then Theorem~\ref{Thrm_ParamSystem_MetricReg} does not follow from the main results of
\cite{Auslender,Borwein}.

Let us also point out that Theorem~\ref{Thrm_ParamSystem_MetricReg} can be easily extended to the case when instead of
quasidifferential sums one uses o.s.c. convexificator mappings. However, since the Clarke subdifferential is 
the smallest o.s.c. convexificator mapping, in the finite dimensional case this result is a corollary to
\cite[Theorem~1.1]{Auslender} as well.
\end{remark}

Let us give an example illustrating Theorem~\ref{Thrm_ParamSystem_MetricReg} and Remark~\ref{Remark_MatrixQuasidiff}.

\begin{example}
Let $X = Y = \mathbb{R}^2$ and $P = \mathbb{R}$. Consider the following system of equations:
\begin{equation} \label{TwoDimSystem_Example}
  \begin{cases}
    \max\{ 2x_1, x_1 \} - | \sin(p x_2) | = y_1, \\
    \sin\big(p (x_1 + x_2) \big) + \min\{ x_2, 2 x_2 \} = y_2.
  \end{cases}
\end{equation}
Define $f_1(x, p) = \max\{ 3x_1, x_1 \} - | \sin(p x_2) |$ and 
$f_2(x, p) = \sin(p (x_1 + x_2)) + \min\{ x_2, 2 x_2 \}$.
Let us utilize Theorem~\ref{Thrm_ParamSystem_MetricReg} to find the values of the parameter $p$ for which the mapping 
$x \mapsto F(x, p) = (f_1(x, p), f_2(x, p))^T$ is metrically regular near the point $(\mathbf{0}_2, \mathbf{0}_2)$. 

The functions $f_1(x, p)$ and $f_2(x, p)$ are quasidifferentiable. With the use of basic rules of quasidifferential
calculus \cite[Section~III.2]{DemRub_book} one obtains that
\begin{align*}
  \underline{\partial}_x f_1(x, p) &= \begin{cases}
    \{ (2, 0)^T \}, & \text{if } x_1 > 0, \\
    \co\{ (1, 0)^T, (2, 0)^T \}, & \text{if } x_1 = 0, \\
    \{ (1, 0)^T \}, & \text{if } x_1 < 0,
  \end{cases}
  \\
  \overline{\partial}_x f_1(x, p) &= 
  \left\{ \begin{pmatrix} 0 \\ - p \cos(p x_2) \svsign\big( \sin(p x_2) \big) \end{pmatrix} \right\},
  \\
  \underline{\partial}_x f_2(x, p) &= \left\{ \begin{pmatrix} p \cos\big( p (x_1 + x_2) \big) \\ 
  p \cos\big( p (x_1 + x_2) \big) \end{pmatrix} \right\}, 
  \\
  \overline{\partial}_x f_2(x, p) &= \begin{cases}
    \{ (0, 1)^T \}, & \text{if } x_2 > 0, \\
    \co\{ (0, 1)^T, (0, 2)^T \}, & \text{if } x_2 = 0, \\
    \{ (0, 2) \}, & \text{if } x_2 < 0.
  \end{cases}
\end{align*}
It is readily seen that the quasidifferential mappings $(x, p) \mapsto \mathscr{D}_x f_1(x, p)$ and 
$(x, p) \mapsto \mathscr{D}_x f_2(x, p)$ are outer semicontinuous. 

Let us verify whether q.d.-MFCQ holds at the point $(\mathbf{0}_2, p)$. Following
Remark~\ref{Remark_MatrixQuasidiff} introduce the quasidifferential 
$\mathscr{D}_x F(\mathbf{0}_2, p) = [ \underline{\partial}_x F(\mathbf{0}_2, p), 
\overline{\partial}_x F(\mathbf{0}_2, p)]$,
\begin{align*}
  \underline{\partial}_x F(\mathbf{0}_2, p) &= \left\{ \begin{pmatrix} t & 0 \\ p & p \end{pmatrix} \biggm| 
  t \in [1, 2] \right\}, \\
  \overline{\partial}_x F(\mathbf{0}_2, p) &= \left\{ \begin{pmatrix} 0 & p t \\ 0 & s \end{pmatrix} \biggm| 
  t \in [-1, 1], s \in [1, 2] \right\},
\end{align*}
of the map $x \mapsto F(x, p)$ at the point $x = \mathbf{0}_2$. The first row of the set 
$\underline{\partial}_x F(\mathbf{0}_2, p)$ corresponds to $\underline{\partial}_x f_1(\mathbf{0}_2, p)$, while 
the second row corresponds to $\underline{\partial}_x f_2(\mathbf{0}_2, p)$. 
The set $\overline{\partial}_x F(\mathbf{0}_2, p)$ is defined in the same way.

The quasidifferential sum of the map $x \mapsto F(x, p)$ at $x = \mathbf{0}_2$ has the form
$$
  [\mathscr{D}_x F(\mathbf{0}_2, p)]^+ = \left\{ \begin{pmatrix} t & ps \\ p & p + r \end{pmatrix}
  \biggm| t \in [1, 2], \: s \in [-1, 1], \: r \in [1, 2] \right\}
$$
Our aim is to find such $p \in \mathbb{R}$ that all matrices from the set $[\mathscr{D}_x F(\mathbf{0}_2, p)]^+$ are
nondegenerate. The determinants of the matrices from $[\mathscr{D}_x F(\mathbf{0}_2, p)]^+$ take values in the set 
$$
  \co\{ 1, 4 \} + \co\{ p, 2p \} + \co\{ -p^2, p^2 \}.
$$
Hence taking into accoun the fact that the determinant of 
$\left(\begin{smallmatrix} 1 & -p \\ p & p + 1  \end{smallmatrix}\right) \in [\mathscr{D}_x F(\mathbf{0}_2, p)]^+$ is
equal to $p^2 + p + 1$ and positive for all $p$ one obtains that $\determ A \ne 0$ for any 
$A \in [\mathscr{D}_x F(\mathbf{0}_2, p)]^+$ iff the following inequalities hold true:
$$
  p^2 + 2p + 1 > 0, \quad - p^2 + p + 1 > 0, \quad - p^2 + 2p + 1 > 0.
$$
Solving these inequalities one obtains that q.d.-MFCQ holds at the point $(\mathbf{0}_2, p)$ iff
$p \in (1 - \sqrt{2}, (1 + \sqrt{5}) / 2)$. Consequently, by Theorem~\ref{Thrm_ParamSystem_MetricReg} one can conclude
that for any $\overline{p} \in (1 - \sqrt{2}, (1 + \sqrt{5}) / 2)$ there exist $K > 0$ and $r > 0$ such that
$$
  d\big( x, (F_p)^{-1}(y) \big) \le K \| y - F(x, p) \|
$$
for all $x, y \in B(\mathbf{0}_2, r)$ and any $p \in (\overline{p} - r, \overline{p} + r)$, which in particular implies
that for any such $y$ and $p$ there exists a solution $x(y, p)$ of system \eqref{TwoDimSystem_Example}.
\end{example}

As the following simple example shows q.d.-MFCQ, unlike MFCQ in the smooth case, is not \textit{necessary} for 
the metric regularity of a multifunction associated with a system of quasidifferentiable equality and inequality
constraints.

\begin{example} \label{Example_qdMFCQ_fails}
Let $X = \mathbb{R}^2$, $Y = \mathbb{R}$, $F(x) = |x_1| - |x_2|$, and suppose that there are no inequality constraints.
Let us check whether q.d.-MFCQ holds at the point $\overline{x} = \mathbf{0}_2$. Indeed, the function $F$ is
quasidifferentiable, and one can define
$$
  \underline{\partial} F(x) = \left\{ \begin{pmatrix} \svsign(x_1) \\ 0 \end{pmatrix} \right\}, \quad
  \overline{\partial} F(x) = \left\{ \begin{pmatrix} 0 \\ -\svsign(x_2) \end{pmatrix} \right\}.
$$
Clearly, the multifunctions $\underline{\partial} F(\cdot)$ and $\overline{\partial} F(\cdot)$ are outer
semicontinuous. Observe that 
$[\mathscr{D} F(\overline{x})]^+ = \{ x \in \mathbb{R}^2 \mid \max\{ |x_1|, |x_2| \} \le 1 \}$, and q.d.-MFCQ is not
satisfied at the origin, since $\mathbf{0}_2 \in [\mathscr{D} F(\overline{x})]^+$, despite the fact that the function
$F$ is metrically regular near the point $(\overline{x}, 0)$ (see Remark~\ref{Remark_ComparisonWithUderzo}).
\end{example}

It should be noted that in the finite dimensional case q.d.-MFCQ imposes some implicit assumptions on the dimension of
the space $X$. For example, if for the system
$$
  f_1(x, p) = y, \quad g_1(x, p) \le 0
$$
the quasidifferential sum $[\mathscr{D}_x f_1(\overline{x}, \overline{p})]^+$ contains at least two linearly independent
vectors, then $\dimension(\linhull [\mathscr{D}_x f_1(\overline{x}, \overline{p})]^+) \ge 2$ and for q.d.-MFCQ to
hold true at $(\overline{x}, \overline{p})$ it is necessary that $\dimension X \ge 3$ 
(see~Remark~\ref{Remark_Equivalent_qdMFCQ}). The following example highlights this drawback of q.d.-MFCQ.

\begin{example}
Let $X = \mathbb{R}^2$, $Y = \mathbb{R}$, and $m = 1$. Consider the following system:
$$
  f(x) = |x_1| - x_2 = y, \quad g(x) = x_1 \le z.
$$
Our aim is to check whether the multifunction $\Phi(x) = \{ f(x) \} \times [g(x), + \infty)$ associated with this system
is metrically regular near the point $(\overline{x}, (0, 0))$ with $\overline{x} = \mathbf{0}_2$. 

Both functions $f$ and $g$ are obviously quasidifferentiable. One can define
\begin{align*}
  \underline{\partial} f(x) &= \left\{ \begin{pmatrix} \svsign(x_1) \\ -1 \end{pmatrix} \right\}, \quad
  \overline{\partial} f(x) = \{ \mathbf{0}_2 \}, \\
  \underline{\partial} g(x) &= \left\{ \begin{pmatrix} 1 \\ 0 \end{pmatrix} \right\}, \qquad
  \overline{\partial} g(x) = \{ \mathbf{0}_2 \}.
\end{align*}
Clearly, the mappings $\mathscr{D} f(\cdot)$ and $\mathscr{D} g(\cdot)$ are outer semicontinuous. Observe that
$$
  [\mathscr{D} f(\overline{x})]^+ = \co\left\{ \begin{pmatrix} 1 \\ -1 \end{pmatrix}, 
  \begin{pmatrix} -1 \\ -1 \end{pmatrix} \right\}, \quad
  [\mathscr{D} g(\overline{x})]^+ = \left\{ \begin{pmatrix} 1 \\ 0 \end{pmatrix} \right\}.
$$
Hence $\linhull [\mathscr{D} f(x)]^+ = \mathbb{R}^2$, which implies that q.d.-MFCQ does not hold at $\overline{x}$, and
Theorem~\ref{Thrm_ParamSystem_MetricReg} cannot be applied. Therefore we utilize
Theorem~\ref{Thrm_MetricRegularity_General} to check whether the multifunction $\Phi$ is metrically regular
near the point $(\overline{x}, (0, 0))$.

Note that
$$
  \psi_{(y, z)}(x) = d( (y, z), \Phi(x) ) = | y - |x_1| + x_2 | + \max\{ 0, x_1 - z \}.
$$
Define $\psi^1_y(x) = |y - |x_1| + x_2|$ and $\psi^2_z(x) = \max\{ 0, x_1 - z \}$. The functions $\psi_{(y, z)}(\cdot)$,
$\psi^1_y(\cdot)$ and $\psi^2_z(\cdot)$ are quasidifferentiable for all $y, z \in \mathbb{R}$. Applying basic rules of
quasidifferential calculus \cite[Section~III.2]{DemRub_book} one obtains that
\begin{align*}
  \underline{\partial} \psi^1_y(x) &= \{ \mathbf{0}_2 \},
  \quad
  \overline{\partial} \psi^1_y(x) = \left\{ \begin{pmatrix} - \svsign(x_1) \\ 1 \end{pmatrix} \right\},
  \quad \text{if } y > f(x), \\
  \underline{\partial} \psi^1_y(x) &= \co\left\{ \begin{pmatrix} 0 \\ 0 \end{pmatrix}, 
  \begin{pmatrix} 2 \svsign(x_1) \\ -2 \end{pmatrix} \right\},
  \:
  \overline{\partial} \psi^1_y(x) = \left\{ \begin{pmatrix} - \svsign(x_1) \\ 1 \end{pmatrix} \right\},
  \: \text{if } y = f(x), \\
  \underline{\partial} \psi^1_y(x) &= \left\{ \begin{pmatrix} \svsign(x_1) \\ -1 \end{pmatrix} \right\},
  \quad
  \overline{\partial} \psi^1_y(x) = \{ \mathbf{0}_2 \},
  \quad \text{if } y < f(x), \\
  \underline{\partial} \psi^2_z(x) &= \begin{cases}
    \{ \mathbf{0}_2 \}, & \text{if } x_1 < z, \\
    \co\{ (0, 0)^T, (1, 0)^T \}, & \text{if } x_1 = z, \\
    \{ (1, 0)^T \}, & \text{if } x_1 > z,
  \end{cases}
  \quad
  \overline{\partial} \psi^2_z(x) = \{ \mathbf{0}_2 \}
\end{align*}
Moreover, 
$\underline{\partial} \psi_{(y, z)}(x) = \underline{\partial} \psi^1_y(x) + \underline{\partial} \psi^2_z(x)$ and
$\overline{\partial} \psi_{(y, z)}(x) = \overline{\partial} \psi^1_y(x) + \overline{\partial} \psi^2_z(x)$.

Fix any $x \in \mathbb{R}^2$ and $y, z \in \mathbb{R}$ such that $(y, z) \notin \Phi(x)$, and suppose that the space 
$X$ is equipped with the Euclidean norm. The following three cases are possible.
\begin{enumerate}
\item{If $y > f(x)$, then for any $t \in \svsign(x_1)$ one has 
$w^* = (-t, 1)^T \in \overline{\partial} \psi_{(y, z)} (x)$ and 
$d(\mathbf{0}_2, \underline{\partial} \psi_{(y, z)}(x) + w^*) \ge 1$, since any 
$v^* \in \underline{\partial} \psi_{(y, z)}(x) + w^*$ has the form $(s, 1)^T$ for some $s \in \mathbb{R}$.
}

\item{If $y < f(x)$, then for $w^* = \mathbf{0}_2 \in \overline{\partial} \psi_{(y, z)}(x)$ one has
$d(\mathbf{0}_2, \underline{\partial} \psi_{(y, z)}(x) + w^*) \ge 1$, since any 
$v^* \in \underline{\partial} \psi_{(y, z)}(x) + w^*$ has the form $(s, -1)^T$ for some $s \in \mathbb{R}$.
}

\item{If $y = f(x)$, then $x_1 > z$ due to the fact that $(y, z) \notin \Phi(x)$. Define
$w^* = (-\sign(x_1), 1) \in \overline{\partial} \psi_{(y, z)}(x)$, if $x_1 \ne 0$, and
$w^* = (1, 1) \in \overline{\partial} \psi_{(y, z)}(x)$, if $x_1 = 0$. Then one can verify that
$d(\mathbf{0}_2, \underline{\partial} \psi_{(y, z)}(x) + w^*) = \sqrt{2} / 2$.
}
\end{enumerate}
Thus, for any $x \in \mathbb{R}^2$ and $y, z \in \mathbb{R}$, $(y, z) \notin \Phi(x)$, there exists 
$w^* \in \overline{\partial} \psi_{(y, z)} (x)$ such that 
$d(\mathbf{0}_2, \underline{\partial} \psi_{(y, z)}(x) + w^*) \ge \sqrt{2} / 2$. Therefore, the multifunction $\Phi$ is
metrically regular near the point $(\overline{x}, (0, 0))$ with the norm of metric regularity not exceeding 
$\sqrt{2} / 2$ by Theorem~\ref{Thrm_MetricRegularity_General}.
\end{example}

\section{Optimality Conditions}
\label{Section_OptimalityConditions}

Let us utilize q.d.-MFCQ as a new constraint qualification for quasidifferential programming problems with equality and
inequality constraints to obtain necessary optimality conditions for these problems. To this end, consider the following
optimization problem:
$$
  \min \: u(x) \quad 
  \text{subject to } f_j(x) = 0, \quad j \in J, \quad g_i(x) \le 0, \quad i \in I.
  \eqno{(\mathcal{P})}
$$
Here $u, f_j, g_i \colon X \to \mathbb{R}$ are given functions, $J = \{ 1, \ldots, l \}$, and $I = \{ 1, \ldots, m \}$.
Our aim is to obtain optimality conditions for the problem $(\mathcal{P})$ via exact penalty function approach. 

Define $\varphi(x) = \sum_{j = 1}^l |f_j(x)| + \sum_{i = 1}^m \max\{ g_i(x), 0 \}$, and denote the $\ell_1$ penalty
function for the problem $(\mathcal{P})$ by $\Psi_c(x) = u(x) + c \varphi(x)$, where $c \ge 0$ is the penalty parameter.
Note that if the functions $u$, $f_j$, and $g_i$ are quasidifferentiable, then this penalty function is
quasidifferentiable as well (see \cite{DemRub_book}). 

Let $\Omega$ be the feasible region of the problem $(\mathcal{P})$, and $\overline{x}$ be a locally optimal solution of
this problem. Observe that $x \in \Omega$ iff $\varphi(x) = 0$. Recall also that if $u$ is Lipschitz continuous near
$\overline{x}$, and the penalty term $\varphi$ has a \textit{local error bound} at $\overline{x}$, i.e. there exists
$\tau > 0$ such that $\varphi(x) \ge \tau d(x, \Omega)$ for any $x$ in a neighbourhood of $\overline{x}$, then 
the penalty function $\Psi_c$ is \textit{locally exact} at $\overline{x}$, i.e. there exist a neighbourhood $U$
of $\overline{x}$ and $c^* \ge 0$ such that 
$$
  \Psi_c(x) \ge \Psi_c(\overline{x}) \quad \forall x \in U \quad \forall c \ge c^*,
$$
(see, e.g. \cite[Theorem~2.4 and Proposition~2.7]{Dolgopolik_ExPen}). If $\Psi_c$ is locally exact at $\overline{x}$,
then by definition $\overline{x}$ is a point of unconstrained local minimum of $\Psi_c$ for any sufficiently large 
$c \ge 0$. In this case one can apply standard necessary conditions for a minimum in terms of quasidifferentials
\cite{DemRub_book} to $\Psi_c$ to obtain necessary optimality conditions for the problem $(\mathcal{P})$.

\begin{theorem} \label{Theorem_QuasidiffProg_OptimCond}
Let the following assumptions be valid:
\begin{enumerate}
\item{$\overline{x}$ is a locally optimal solution of the problem $(\mathcal{P})$;}

\item{$u$ is quasidifferentiable at $\overline{x}$ and Lipschitz continuous near this point;}

\item{$f_j$, $j \in J$, and $g_i$, $i \in I$, are quasidifferentiable in a neighbourhood of $\overline{x}$, and there
exist quasidifferential mappings $\mathscr{D} f_j(\cdot)$, $j \in J$, and $\mathscr{D} g_i(\cdot)$, $i \in I$, defined
in a neighbourhood of $\overline{x}$ and o.s.c. at this point;}

\item{q.d.-MFCQ holds at $\overline{x}$.}
\end{enumerate}
Then there exists $c^* \ge 0$ such that for any $c \ge c^*$ one has
\begin{equation} \label{QuasidiffProg_OptCond_via_ExPenFunc}
  \mathbb{O} \in \underline{\partial} \Psi_c(\overline{x}) + w^*
  \quad \forall w^* \in \overline{\partial} \Psi_c(\overline{x}),
\end{equation}
where $\mathscr{D} \Psi_c(\overline{x}) = [ \underline{\partial} \Psi_c(\overline{x}), 
\overline{\partial} \Psi_c(\overline{x})]$ is any quasidifferential of $\Psi_c$ at $\overline{x}$. 
Moreover, for any $w_0^* \in \overline{\partial} u(\overline{x})$, 
$v_j^* \in \underline{\partial} f_j(\overline{x})$,
$w_j^*  \in \overline{\partial} f_j(\overline{x})$, $j \in J$, and 
$z_i^* \in \overline{\partial} g_i(\overline{x})$, $i \in I$, there exist
$\underline{\mu}_j, \overline{\mu}_j, \lambda_i \ge 0$ such that $\lambda_i g_i(\overline{x}) = 0$ for all $i \in I$ and
\begin{equation} \label{QuasidiffProg_LagrangeMultipliers}
\begin{split}
  \mathbb{O} \in \underline{\partial} u(\overline{x}) + w_0^* 
  &- \sum_{j = 1}^l \underline{\mu}_j \Big( v_j^* + \overline{\partial} f_j(\overline{x}) \Big) \\
  &+ \sum_{j = 1}^l \overline{\mu}_j \Big( \underline{\partial} f_j(\overline{x}) + w_j^* \Big)
  + \sum_{i = 1}^m \lambda_i \Big( \underline{\partial} g_i(\overline{x}) + z_i^* \Big).
\end{split}
\end{equation}
In addition, one can choose $\underline{\mu}_j, \overline{\mu}_j$, and $\lambda_i$ in such a way that for all $i \in I$
and $j \in J$ one has $\max\{ \underline{\mu}_j + \overline{\mu}_j, \lambda_i \} \le c^*$ , i.e.
the multipliers $\underline{\mu}_j, \overline{\mu}_j$, and $\lambda_i$ are bounded for all 
$w_0^* \in \overline{\partial} u(\overline{x})$,
$v_j^* \in \underline{\partial} f_j(\overline{x})$,
$w_j^*  \in \overline{\partial} f_j(\overline{x})$, $j \in J$, and 
$z_i^* \in \overline{\partial} g_i(\overline{x})$, $i \in I$.
\end{theorem}

\begin{proof}
Let us show at first that q.d.-MFCQ guarantees that $\varphi$ has a local error bound. Suppose that $\mathbb{R}^l$ is
endowed with the Euclidean norm. If q.d.-MFCQ holds at $\overline{x}$, then by Theorem~\ref{Thrm_ParamSystem_MetricReg}
the multifunction $\Phi \colon X \to \mathbb{R}^l \times \mathbb{R}^m$,
$\Phi(x) = \prod_{j = 1}^l \{ f_j(x) \} \times \prod_{i = 1}^m [ g_i(x), + \infty)$ is metrically regular near the point
$(\overline{x}, (\mathbf{0}_l, \mathbf{0}_m))$. Hence, in particular, there exist $K > 0$ and a neighbourhood $U$ of
$\overline{x}$ such that
$$
  d(x, \Omega) = d\big(x, \Phi^{-1}(\mathbf{0}_l, \mathbf{0}_m) \big) 
  \le K d((\mathbf{0}_l, \mathbf{0}_m), \Phi(x)) \le K \varphi(x)
$$
for all $x \in U$, i.e. $\varphi$ has a local error bound at $\overline{x}$.

Now we can turn to the proof of \eqref{QuasidiffProg_OptCond_via_ExPenFunc}. Under the assumptions of the theorem the
penalty function $\Psi_c$ is locally exact at $\overline{x}$ by \cite[Theorem~2.4 and
Proposition~2.7]{Dolgopolik_ExPen}. Thus, there exists $c^* \ge 0$ such that for any $c \ge c^*$ the point
$\overline{x}$ is a local minimizer of $\Psi_c$. Consequently, applying the necessary conditions for a minimum in
terms of quasidifferentials \cite[Theorem~V.3.1]{DemRub_book} to $\Psi_c$ one gets that 
$\mathbb{O} \in \underline{\partial} \Psi_c(\overline{x}) + w^*$ for 
all $w^* \in \overline{\partial} \Psi_c(\overline{x})$, i.e. \eqref{QuasidiffProg_OptCond_via_ExPenFunc} holds true.

To prove the validity of \eqref{QuasidiffProg_LagrangeMultipliers} note that by the necessary condition for a minimum
in terms of directional derivative for all $c \ge c^*$ and $h \in X$ one has 
$$
  \Psi'_c(\overline{x}, h) = u'(\overline{x}, h) +
  c \Big( \sum_{j = 1}^l \big| f'_j(\overline{x}, h) \big|
  + \sum_{i \in I(\overline{x})} \max\big\{ g'_j(\overline{x}, h), 0 \big\} \Big)
  \ge 0,
$$
where $I(\overline{x}) = \{ i \in I \mid g_i(\overline{x}) = 0 \}$ (here we used standard calculus rules for directional
derivatives; see, e.g.~\cite[Sect.~I.3]{DemRub_book}). Let $w_0^*$, $v_j^*$, $w_j^*$ and $z_i^*$
be as in the formulation of the theorem. Define $s(C, h) = \sup_{x^* \in C} \langle x^*, h \rangle$ for any 
$C \subset X^*$, and
\begin{align*}
  \xi_c(h) = s(\underline{\partial} u(\overline{x}) + w_0^*, h)
  &+ c \sum_{j = 1}^l \max\Big\{ s(\underline{\partial} f_j(\overline{x}) + w_j^*, h),
  s(- v_j^* - \overline{\partial} f_j(\overline{x}), h) \Big\} \\
  &+ c \sum_{i \in I(\overline{x})} \max\Big\{ s(\underline{\partial} g_i(\overline{x}) + z_i^*, h), 0 \Big\}
  \quad \forall h \in X.
\end{align*}
Applying the definition of quasidifferential it is easy to see that $\xi_c(h) \ge \Psi'_c(x, h) \ge 0$ for all 
$c \ge c^*$ and $h \in X$. Therefore, $\mathbb{O}$ is a point of global minimum of the function $\xi_c$, since
$\xi_c(\mathbb{O}) = 0$, which implies that $\mathbb{O} \in \partial \xi_c(\mathbb{O})$ for any $c \ge c^*$, where
$\partial \xi_c(\mathbb{O})$ is the subdifferential of $\xi_c$ at $\mathbb{O}$ in the sense of convex analysis. Applying
standard calculus rules for subdifferentials of convex functions one obtains that
\begin{align*}
  \mathbb{O} \in \partial \xi_c(\mathbb{O}) = \underline{\partial} u(\overline{x}) + w_0^*
  &+ c \sum_{j = 1}^l \co\Big\{ \underline{\partial} f_j(\overline{x}) + w_j^*,
  - v_j^* - \overline{\partial} f_j(\overline{x}) \Big\} \\
  &+ c \sum_{i \in I(\overline{x})} \co\Big\{ \underline{\partial} g_i(\overline{x}) + z_i^*, \mathbb{O} \Big\}.
\end{align*}
for all $c \ge c^*$. Hence for any $c \ge c^*$ there exists $\alpha_j \in [0, 1]$, $j \in J$, and $\beta_i \in [0, 1]$,
$i \in I(\overline{x})$, such that \begin{align*}
  \mathbb{O} \in \underline{\partial} u(\overline{x}) + w_0^* 
  &+ c \sum_{j = 1}^l \alpha_j \Big( \underline{\partial} f_j(\overline{x}) + w_j^* \Big) \\
  &- c \sum_{j = 1}^l (1 - \alpha_j) \Big( v_j^* + \overline{\partial} f_j(\overline{x}) \Big)   
  + c \sum_{i = 1}^m \beta_i \Big( \underline{\partial} g_i(\overline{x}) + z_i^* \Big).
\end{align*}
Denoting $\underline{\mu}_j = c (1 - \alpha_j)$, $\overline{\mu}_j = c \alpha_j$, $j \in J$, 
$\lambda_i = c \beta_i$ for $i \in I(\overline{x})$, and $\lambda_i = 0$ for $i \in I \setminus I(\overline{x})$ one
obtains that \eqref{QuasidiffProg_LagrangeMultipliers} holds true. Note finally that setting $c = c^*$ one gets the
required bound on multipliers.  
\end{proof}

\begin{remark} \label{Remark_ErrorBoundInsteadOfMFCQ}
Note that in the theorem above instead of q.d.-MFCQ it is sufficient to suppose that the penalty term $\varphi$ has a
local error bound at $\overline{x}$.
\end{remark}

\begin{remark}
Optimality conditions similar to but weaker than \eqref{QuasidiffProg_OptCond_via_ExPenFunc} were obtained in
\cite{Shapiro84,Shapiro86} in the finite dimensional case under a different constraint qualification that involves some
assumptions on so-called \textit{contact points} of the sets $\underline{\partial} f_j(\overline{x})$ and 
$\overline{\partial} f_j(\overline{x})$, i.e. such points $v^*$ of a convex set $C \subset X^*$ that 
$s(C, h) = \langle v^*, h \rangle$ for a given direction $h$. Note that one has to compute contact
points of the sets $\underline{\partial} f_j(\overline{x})$ and $\overline{\partial} f_j(\overline{x})$ for \textit{all}
feasible directions in order to check the validity of the constraint qualification from \cite{Shapiro84,Shapiro86},
which is impossible in nontrivial cases. In contrast, q.d.-MFCQ is formulated in terms of problem data directly. In
turn, optimality conditions similar to but weaker than \eqref{QuasidiffProg_LagrangeMultipliers} were derived in 
\cite{Polyakova} under yet another constraint qualification in the case when $X$ is finite dimensional, there
are no inequality constraints, and there is only one equality constraint. Furthermore, note that sufficient conditions
for the validity of this constraint qualification \cite[Theorem~2]{Polyakova} coincide with q.d.-MFCQ with 
$I = \emptyset$ and $l = 1$.
\end{remark}

At first glance optimality condition \eqref{QuasidiffProg_OptCond_via_ExPenFunc} might seem sharper than condition
\eqref{QuasidiffProg_LagrangeMultipliers}. Let us show that these conditions are in fact equivalent and independent of
the choice of quasidifferentials (cf.~\cite{Luderer,LudererRosigerWurker}).

\begin{proposition}
Let the functions $u$, $f_j$, $j \in J$, and $g_i$, $i \in I$, be quasidifferentiable at a feasible point
$\overline{x}$ of the problem $(\mathcal{P})$. Then \eqref{QuasidiffProg_OptCond_via_ExPenFunc} is satisfied for some 
$c \ge 0$ if and only if for any $w_0^* \in \overline{\partial} u(\overline{x})$, 
$v_j^* \in \underline{\partial} f_j(\overline{x})$,
$w_j^*  \in \overline{\partial} f_j(\overline{x})$, $j \in J$, and 
$z_i^* \in \overline{\partial} g_i(\overline{x})$, $i \in I$, 
there exist $\underline{\mu}_j, \overline{\mu}_j, \lambda_i \ge 0$ such that \eqref{QuasidiffProg_LagrangeMultipliers}
holds true, and for all $i \in I$ and $j \in J$ one has 
$\lambda_i g_i(\overline{x}) = 0$ and $\max\{ \underline{\mu}_j + \overline{\mu}_j, \lambda_i \} \le c$.
Furthermore, both these conditions are independent of the choice of corresponding quasidifferentials. 
\end{proposition}

\begin{proof}
From the definition of quasidifferential it follows that
$$
  \Psi'_c(\overline{x}, h) = \min_{w^* \in \overline{\partial} \Psi_c (\overline{x})} 
  \max_{v^* \in \underline{\partial} \Psi_c(\overline{x}) + w^*} \langle v^*, h \rangle 
  \quad \forall h \in X,
$$
which implies that \eqref{QuasidiffProg_OptCond_via_ExPenFunc} is satisfied for some $c \ge 0$ iff 
$\Psi'_c(\overline{x}, h) \ge 0$ for all $h \in X$. The latter condition is obviously independent of the choise
of quasidifferential. Therefore optimality condition \eqref{QuasidiffProg_OptCond_via_ExPenFunc} is independent of the
choice of a quasidifferential of $\Psi_c$ as well. 

Let us now show that optimality conditions \eqref{QuasidiffProg_OptCond_via_ExPenFunc} and
\eqref{QuasidiffProg_LagrangeMultipliers} are equivalent. Indeed, let \eqref{QuasidiffProg_OptCond_via_ExPenFunc} be
valid for some quasidifferential of $\Psi_c$ at $\overline{x}$ and $c \ge 0$. Then $\Psi'_c(x, h) \ge 0$ for all 
$h \in X$. Hence arguing in the same way as in the proof of Theorem~\ref{Theorem_QuasidiffProg_OptimCond} one obtains
that for any $w_0^* \in \overline{\partial} u(\overline{x})$, 
$v_j^* \in \underline{\partial} f_j(\overline{x})$,
$w_j^*  \in \overline{\partial} f_j(\overline{x})$, $j \in J$, and 
$z_i^* \in \overline{\partial} g_i(\overline{x})$, $i \in I$, 
there exist $\underline{\mu}_j, \overline{\mu}_j, \lambda_i \ge 0$ such that \eqref{QuasidiffProg_LagrangeMultipliers}
holds true, and for all $i \in I$, $j \in J$ one has $\lambda_i g_i(\overline{x}) = 0$ and 
$\max\{ \underline{\mu}_j + \overline{\mu}_j, \lambda_i \} \le c$. Note that the
implication $\eqref{QuasidiffProg_OptCond_via_ExPenFunc} \implies \eqref{QuasidiffProg_LagrangeMultipliers}$ is valid
for any quasidifferentials of the functions $u$, $f_i$, and $g_j$.

Let us prove the converse implication. Fix any quasidifferentials of the functions $u$, $f_i$, and $g_j$, and suppose
that there exists $c_0 \ge 0$ such that for any $w_0^* \in \overline{\partial} u(\overline{x})$, 
$v_j^* \in \underline{\partial} f_j(\overline{x})$,
$w_j^*  \in \overline{\partial} f_j(\overline{x})$, $j \in J$, and 
$z_i^* \in \overline{\partial} g_i(\overline{x})$, $i \in I$,
there exist $\underline{\mu}_j, \overline{\mu}_j, \lambda_i \ge 0$ such that \eqref{QuasidiffProg_LagrangeMultipliers}
holds true, and for all $i \in I$, $j \in J$ one has $\lambda_i g_i(\overline{x}) = 0$ and
$\max\{ \underline{\mu}_j + \overline{\mu}_j, \lambda_i \} \le c_0$. 

Arguing by reductio ad absurdum suppose that \eqref{QuasidiffProg_OptCond_via_ExPenFunc} does not hold true 
for $c = c_0$. Then there exists $h_0 \in X$ such that $\Psi'_{c_0}(\overline{x}, h_0) < 0$. Applying standard calculus
rules for directional derivatives (see, e.g.~\cite[Sect.~I.3]{DemRub_book}) one obtains that
\begin{equation} \label{DecentDirect_L1PenaltyFunc}
\begin{split}
  \Psi'_{c_0}(\overline{x}, h_0) = u'(\overline{x}, h_0) +
  c_0 \Big( &\sum_{j = 1}^l \max\{ f'_j(\overline{x}, h_0), -f'_j(\overline{x}, h_0) \} \\
  + &\sum_{i \in I(\overline{x})} \max\big\{ g'_j(\overline{x}, h_0), 0 \big\} \Big)
  < 0.
\end{split}
\end{equation}
By the definition of quasidifferential there exist $w_0^* \in \overline{\partial} u(\overline{x})$,
$v_j^* \in \underline{\partial} f_j(\overline{x})$,
$w_j^*  \in \overline{\partial} f_j(\overline{x})$, $j \in J$, and 
$z_i^* \in \overline{\partial} g_i(\overline{x})$, $i \in I(\overline{x})$, such that
\begin{align*}
  u'(\overline{x}, h_0) &= \max_{v^* \in \underline{\partial} u(\overline{x})} \langle v^*, h_0 \rangle
  + \langle w_0^*, h_0 \rangle, \\
  f_j'(\overline{x}, h_0) &= \max_{v^* \in \underline{\partial} f_j(\overline{x})} \langle v^*, h_0 \rangle
  + \langle w_j^*, h_0 \rangle,	\quad \forall j \in J, \\
  f_j'(\overline{x}, h_0) &= \langle v_j^*, h_0 \rangle +
  \min_{w^* \in \underline{\partial} f_j(\overline{x})} \langle w^*, h_0 \rangle,	\quad \forall j \in J, \\
  g_i'(\overline{x}, h_0) &= \max_{v^* \in \underline{\partial} g_i(\overline{x})} \langle v^*, h_0 \rangle
  + \langle z_i^*, h_0 \rangle,	\quad \forall i \in I(\overline{x}).
\end{align*}
Hence taking into account \eqref{DecentDirect_L1PenaltyFunc} one obtains that $\xi_{c_0}(h_0) < 0$, where the function
$\xi_c$ is defined in the proof of Theorem~\ref{Theorem_QuasidiffProg_OptimCond}. On the other hand, from the validity
of \eqref{QuasidiffProg_LagrangeMultipliers} with 
$\max\{ \underline{\mu}_j + \overline{\mu}_j, \lambda_i \} \le c_0$ it follows that 
$\mathbb{O} \in \partial \xi_{c_0}(\mathbb{O})$ (see the proof of Theorem~\ref{Theorem_QuasidiffProg_OptimCond}).
Therefore $\xi_{c_0}(h) \ge \xi_{c_0}(\mathbb{O}) = 0$ for all $h \in X$, which contradicts the inequality
$\xi_{c_0}(h_0) < 0$. Thus, \eqref{QuasidiffProg_OptCond_via_ExPenFunc} holds true for $c = c_0$.

Let us finally show the independence of \eqref{QuasidiffProg_LagrangeMultipliers} on the choice of quasidifferentials.
Indeed, if \eqref{QuasidiffProg_LagrangeMultipliers} is valid for one choice of quasidifferentials of the functions $u$,
$f_j$, and $g_i$, then, as we have just proved, optimality condition \eqref{QuasidiffProg_OptCond_via_ExPenFunc} is
satisfied. Hence with the use of the implication 
$\eqref{QuasidiffProg_OptCond_via_ExPenFunc} \implies \eqref{QuasidiffProg_LagrangeMultipliers}$ one obtains that 
\eqref{QuasidiffProg_LagrangeMultipliers} is valid for any other choice of quasidifferentials of the functions $u$,
$f_j$, and $g_i$.	 
\end{proof}

Let us also give a simple example demonstrating that in some cases the optimality conditions from
Theorem~\ref{Theorem_QuasidiffProg_OptimCond} are much sharper than optimality conditions in terms of various
subdifferentials.

\begin{example}
Let $X = \mathbb{R}^2$, and consider the following optimization problem:
\begin{equation} \label{ExampleProblem}
  \min u(x) = - x_1 + x_2 \quad \text{subject to} \quad f_1(x) = |x_1| - |x_2| = 0.
\end{equation}
Put $\overline{x} = \mathbf{0}_2$. Observe that $\overline{x}$ is not a locally optimal solution of problem
\eqref{ExampleProblem}, since for any $t > 0$ the point $x(t) = (t, -t)$ is feasible for this problem and
$u(x(t)) = -2t < 0 = u(\overline{x})$. Nevertheless, let us verify that several subdifferential-based optimality
conditions fail to disqualify $\overline{x}$ as a non-optimal solution.

We start with necessary optimality conditions in terms of the subdifferential of Michel-Penot \cite{Ioffe93}, which we
denote by $\partial_{MP}$. Let $L(x, \lambda) = u(x) + \lambda f_1(x)$ be the Lagrangian function for problem
\eqref{ExampleProblem}. For any $h \in \mathbb{R}^2$ the Michel-Penot directional derivative of $L(\cdot, \lambda)$ at
$\overline{x}$ has the form
\begin{multline*}
  d_{MP} L(\cdot, \lambda)[\overline{x}, h] = \sup_{e \in \mathbb{R}^2}
  \limsup_{t \to + 0} \frac{L(x + t(h + e)) - L(x + te)}{t} \\
  = \sup_{e \in \mathbb{R}^2} \Big\{ -h_1 + h_2 + \lambda\Big(|h_1 + e_1| - |e_1| - |h_2 + e_2| + |e_2|\Big) \Big\}
  = -h_1 + h_2 + |\lambda| \Big( |h_1| + |h_2| \Big).
\end{multline*}
Hence the Michel-Penot subdifferential of $L(\cdot, \lambda)$ at $\overline{x}$ has the form
$$
  \partial_{MP} L(\cdot, \lambda)(\overline{x}) = \partial \varphi(\mathbf{0}_2) = 
  \co\left\{ \begin{pmatrix} |\lambda| - 1 \\ |\lambda| + 1 \end{pmatrix} 
  \begin{pmatrix} |\lambda| - 1 \\  - |\lambda| + 1 \end{pmatrix}
  \begin{pmatrix} - |\lambda| - 1 \\ |\lambda| + 1 \end{pmatrix}
  \begin{pmatrix} - |\lambda| - 1 \\ - |\lambda| + 1 \end{pmatrix}
  \right\}
$$
where $\varphi(h) = d_{MP} L(\cdot, \lambda)[\overline{x}, h]$. Consequently, for any $\lambda \in \mathbb{R}$ such
that $|\lambda| \ge 1$ one has $\mathbf{0}_2 \in \partial_{MP} L(\cdot, \lambda)(\overline{x})$, which implies that the
optimality conditions from \cite{Ioffe93} are satisfied at $\overline{x}$. Furthermore, note that 
$\partial_{MP} L(\cdot, \lambda)(\overline{x}) = \partial_{Cl} L(\cdot, \lambda)(\overline{x})$, which implies that
optimality conditions in terms of the Clarke subdifferential \cite[Theorem~6.1.1]{Clarke} are satisfied at
$\overline{x}$ for any $\lambda$ with $|\lambda| \ge 1$ as well.

Next, we consider optimality conditions in term of the Jeyakumar-Luc sub\-differential \cite{WangJeyakumar}, which we
denote by $\partial_{JL}$. By \cite[Example~2.1]{WangJeyakumar} one has 
$\partial_{JL} f_1(\overline{x}) = \{ (1, -1)^T, (-1, 1)^T \}$, and clearly 
$\partial_{JL} u(\overline{x}) = \{ (-1, 1)^T \}$. Hence for any $\lambda \in \mathbb{R}$ with $|\lambda| \ge 1$
one has $\mathbf{0}_2 \in \partial_{JL} u(\overline{x}) + \lambda \co \partial_{JL} f_1(\overline{x})$, i.e. the
optimality
conditions in terms of the Jeyakumar-Luc subdifferential \cite[Corollary~3.4]{WangJeyakumar} are satisfied at
$\overline{x}$.

Let us now consider optimality conditions in terms of approximate (graded, Ioffe) subdifferentials
(see~\cite{Ioffe84,Ioffe2012,Penot_book}), which we denote by $\partial_a$. Observe that for any $x \in \mathbb{R}^2$
such that $x_1, x_2 > 0$ one has $L(x, 1) = 0$, which obviously implies that $\partial^-_x L(x, 1) = \{ \mathbf{0}_2 \}$
for any such $x$, where $\partial^-_x L(x, 1)$ is the Dini subdifferential of $L(\cdot, 1)$ at $x$. 
Therefore, $\mathbf{0}_2 \in \partial_a L(\cdot, 1)(\overline{x}) = \limsup_{x \to \overline{x}} \partial^-_x L(x, 1)$,
i.e. the optimality conditions in terms of approximate subdifferential \cite[Proposition~12]{Ioffe84} are satisfied at
$\overline{x}$ (here $\limsup$ is the outer limit). 

Let us also consider optimality conditions in terms of the Mordukhovich basic subdifferential \cite{Mordukhovich_II},
which we denote by $\partial_M$. One can check (see~\cite[p.~92--93]{Mordukhovich_I}) that
$$
  \partial_M f_1(\overline{x}) = \co\left\{ \begin{pmatrix} 1 \\ -1 \end{pmatrix},
  \begin{pmatrix} -1 \\ -1 \end{pmatrix} \right\}
  \cup \co\left\{ \begin{pmatrix} 1 \\ 1 \end{pmatrix}, \begin{pmatrix} -1 \\ 1 \end{pmatrix} \right\}.
$$
Therefore, $- \nabla u(\overline{x}) \in \partial_M f_1(\overline{x})$, i.e. the optimality conditions in terms of the
Mordukhovich basic subdifferential \cite[Theorem~5.19]{Mordukhovich_II} hold true at $\overline{x}$.

Finally, let us verify that optimality conditions \eqref{QuasidiffProg_LagrangeMultipliers} from
Theorem~\ref{Theorem_QuasidiffProg_OptimCond} are not satisfied at $\overline{x}$, i.e. unlike optimality conditions in
terms of various subdifferentials, optimality conditions based on quasidifferentials detect the non-optimality 
of $\overline{x}$. 

Arguing by reductio ad absurdum, suppose that \eqref{QuasidiffProg_LagrangeMultipliers} holds true. Then for 
$v_1^* = (1, 0)^T \in \underline{\partial} f_1(\overline{x})$ and
$w_1^* = (0, 1)^T \in \overline{\partial} f_1(\overline{x})$ 
(see Example~\ref{Example_qdMFCQ_fails}) there exist $\underline{\mu}_1, \overline{\mu}_1 \ge 0$ such that
$$
  0 \in \begin{pmatrix} -1 \\ 1 \end{pmatrix}
  - \underline{\mu}_1 \co\left\{ \begin{pmatrix} 1 \\ -1 \end{pmatrix}, 
  \begin{pmatrix} 1 \\ 1 \end{pmatrix} \right\}
  + \overline{\mu}_1 \co\left\{ \begin{pmatrix} -1 \\ 1 \end{pmatrix},
  \begin{pmatrix} 1 \\ 1 \end{pmatrix} \right\},
$$
or equivalently
$$
  -1 - \underline{\mu}_1 - \overline{\mu}_1 \le 0 \le -1 - \underline{\mu}_1 + \overline{\mu}_1, \quad
  1 + \overline{\mu}_1 - \underline{\mu}_1 \le 0 \le 1 + \overline{\mu}_1 + \underline{\mu}_1.
$$
From the third inequality it follows that $1 + \overline{\mu_1} \le \underline{\mu}_1$, while from the second
inequality it follows that $1 + \underline{\mu}_1 \le \overline{\mu}_1$. Therefore 
$2 + \overline{\mu}_1 \le \overline{\mu}_1$, which is impossible. Thus, optimality conditions
\eqref{QuasidiffProg_LagrangeMultipliers} do not hold true at $\overline{x}$.

As was shown in Remark~\ref{Remark_ComparisonWithUderzo}, the function $f_1$ is metrically regular near the point
$(\overline{x}, 0)$, which obviously implies that the penalty term $\varphi(x) = |f_1(x)|$ has a local error bound at
$\overline{x}$. Therefore, by Theorem~\ref{Theorem_QuasidiffProg_OptimCond} and
Remark~\ref{Remark_ErrorBoundInsteadOfMFCQ} one can conclude that optimality conditions
\eqref{QuasidiffProg_LagrangeMultipliers} are not satisfied at $\overline{x}$ due to the non-optimality of this point.
\end{example}

\bibliographystyle{abbrv}  
\bibliography{DolgopolikMV_bibl}

\end{document}